\documentclass[11pt]{amsart}
\usepackage{cite,amssymb,mathrsfs,bbm,stmaryrd,bbm,hyperref}
\usepackage[english]{babel}
\usepackage[utf8]{inputenc}

\usepackage[headings]{fullpage}

\newtheorem{thm}{Theorem}[section]
 
 \newtheorem{lem}[thm]{Lemma}
 \newtheorem{prop}[thm]{Proposition}

 \theoremstyle{definition}
 \newtheorem{df}[thm]{Definition}
 \newtheorem{ex}[thm]{Example}
 \theoremstyle{remark}
 \newtheorem{rem}[thm]{Remark}
 \numberwithin{equation}{section}

\def\be#1 {\begin{equation} \label{#1}}
\newcommand{\ee}{\end{equation}}

\def\sqw{\hbox{\rlap{\leavevmode\raise.3ex\hbox{$\sqcap$}}$%
\sqcup$}}
\def\findem{\ifmmode\sqw\else{\ifhmode\unskip\fi\nobreak\hfil
\penalty50\hskip1em\null\nobreak\hfil\sqw
\parfillskip=0pt\finalhyphendemerits=0\endgraf}\fi}

\newcommand{\R}{\mathbb R}

\newcommand{\N}{\mathbb N}

\newcommand{\C}{\mathbb C}

\newcommand{\T}{{\mathbb T}}

\newcommand\<{\langle}
\renewcommand\>{\rangle}

\usepackage{cite}
\renewcommand{\Re}{\operatorname{Re}}
\renewcommand{\Im}{\operatorname{Im}}
\renewcommand{\epsilon}{\varepsilon}

\newcommand{\mbf}{\mathbf}

\DeclareMathOperator{\dist}{dist}

\usepackage{pgfplots}
\newcommand{\bg}{g}

\newcommand{\bz}{z}
\newcommand{\bff}{f}
\usepackage{mathrsfs}

\newcommand{\Cont}{C}
\newcommand{\Test}{\Cont^\infty_c}

\newcommand{\cC}{\mathcal C}
\newcommand{\cS}{\mathcal S}
\newcommand{\cT}{\mathcal T}

\newcommand{\bT}{\mbf T}
\newcommand{\bPi}{\mbf \Pi}
\newcommand{\bbo}{\mathbbm 1}
\DeclareMathOperator{\sech}{sech}

\newcommand{\normm}[1]{{\vert\kern-0.25ex\vert\kern-0.25ex\vert #1 \vert\kern-0.25ex\vert\kern-0.25ex\vert}}
\newcommand{\TN}[2]{{\vert\kern-0.25ex\vert\kern-0.25ex\vert #1 \vert\kern-0.25ex\vert\kern-0.25ex\vert}_{#2}}

\newcommand{\sbrack}[1]{^{(#1)}}

\newcommand{\qtq}[1]{\quad\text{#1}\quad}

\newcommand{\bB}{b}

\newcommand{\BMO}{\mathrm{BMO}}
\newcommand{\LHS}[1]{\mathrm{LHS}\eqref{#1}}
\newcommand{\RHS}[1]{\mathrm{RHS}\eqref{#1}}
\usepackage{bm}

\title[A quasilinear Schr\"odinger equation with degenerate dispersion]{Local well-posedness for a quasilinear Schr\"odinger equation with degenerate dispersion}
\author{Benjamin Harrop-Griffiths}
\email{harropgriffiths@math.ucla.edu}
\address{Department of Mathematics, UCLA\\
520 Portola Plaza, Los Angeles, CA 90095, USA}
\author{Jeremy L.~Marzuola}
\email{marzuola@math.unc.edu}
\address{Mathematics Department, University of North Carolina \\
Phillips Hall, Chapel Hill, NC 27599, USA}
\thanks{J.L.M. was supported in part by U.S. NSF Grants DMS--1312874 and DMS--1352353.}

\begin{document}

\maketitle

\begin{abstract}
We consider a quasilinear Schr\"odinger equation on $\R$ for which the dispersive effects degenerate when the solution vanishes. We first prove local well-posedness for sufficiently smooth, spatially localized, degenerate initial data. As a corollary in the focusing case we obtain a short time stability result for the energy-minimizing compact breather.
\end{abstract}

\section{Introduction}

We consider solutions \(u\colon \R_t\times \R_x\rightarrow \C\) of the quasilinear Schr\"odinger equation
\begin{equation}\label{QLS}\tag{QLS}
\begin{cases}
iu_t = \bar u(uu_x)_x + \mu |u|^2 u,\smallskip\\
u(0,x) = u_0(x),
\end{cases}
\end{equation}
where \(\mu \in\{-1,0,1\}\). Our interest in the model \eqref{QLS} originated with the article~\cite{MR2651381}, where the authors reduce the study of norm growth for the defocusing NLS on \(\T^2\) to a discrete toy model. The equation \eqref{QLS} (with \(\mu = 1\)) then arises as a formal continuum limit of this toy model (see~\cite{MR3171091}).

The equation \eqref{QLS} is the Hamiltonian flow of
\[
H[u] := \int|uu_x|^2\,dx - \frac\mu 2\int |u|^4\,dx,
\]
with respect to the Poisson structure
\[
\{F,H\} := i\int \frac{\delta F}{\delta u}\frac{\delta H}{\delta \bar u} - \frac{\delta F}{\delta \bar u}\frac{\delta H}{\delta u}\,dx.
\]
Solutions of \eqref{QLS} also (formally) conserve the mass
\[
M[u] := \int |u|^2\,dx,
\]
and momentum
\[
P[u] := \Im \int u \bar u_x\,dx.
\]

In this article we are primarily interested in the local well-posedness of \eqref{QLS}. Taking \(w = d u\), the linearization of the equation \eqref{QLS} about a solution \(u\) may be written as
\[
i(w_t + vw_x) = (\rho w_x)_x + \text{lower order},
\]
where the density \(\rho\) and velocity\footnote{Strictly speaking, \(v\) is twice the momentum density. However, a peculiarity of \eqref{QLS} is that the mass \(\rho\) is transported by \(v\), which motivates us to refer to it as the velocity.} \(v\) are defined by
\begin{equation}\label{HydroVariables}
\rho := |u|^2,\qquad v := 2\Im(u\bar u_x).
\end{equation}
In particular, the linearized problem is dispersive whenever \(\rho(t,x)\gtrsim 1\). On sufficiently short time intervals we expect that \(\rho(t,x)\approx \rho(0,x) = |u_0(x)|^2\), so the dispersive nature of the problem is determined by the initial data. If the initial data \(u_0\) is non-degenerate, i.e. \(|u_0(x)|\gtrsim 1\), local well-posedness then follows from e.g.~\cite{MR2955206,MR2096797,marzuola2020quasilinear}. Unfortunately, these techniques break down when \(u_0\) is allowed to degenerate.

We will focus on the case that the initial data \(u_0\) is smooth and non-zero on the interval \(I : = (-x_0,x_0)\subset \R\) and supported on the closed interval \(\bar I = [-x_0,x_0]\), with sufficient decay at the endpoints to ensure that
\begin{equation}\label{DecayRate}
\frac1{|u_0|}\not \in L^1\big((-x_0,0)\big)\cup L^1\big((0,x_0)\big).
\end{equation}
A particular example to keep in mind is the case that \(|u_0|\) has asymptotic behavior
\begin{equation}\label{endpointasymptotics}
|u_0(x)| \sim c_\pm \dist(x,\pm x_0)^{1 + \alpha_\pm}\qtq{as} x\rightarrow(\pm x_0)^\mp,
\end{equation}
for positive constants \(c_\pm>0\) and non-negative constants \(\alpha_\pm\geq0\). We refer to the case \(\alpha_\pm = 0\) as linear endpoint decay and \(\alpha_\pm>0\) as sublinear\footnote{Since we are considering compact regions, our notion of sublinear and superlinear is reversed from behavior considered as $x \to \infty$.} endpoint decay. Heuristically, we expect linear endpoint decay to be sharp in the sense that linear and sublinear endpoint decay will lead to local well-posedness in suitable spaces, whereas superlinear (\(\alpha_\pm<0\)) endpoint decay will lead to ill-posedness in any reasonable space of distributions. These heuristics are derived by linearizing the equation \eqref{QLS} about the initial data and considering the Hamiltonian flow of the corresponding principal symbol (see the discussion in~\cite{MR4011864} for example).

While preparing this article, we learned of a complementary preprint by Jeong and Oh~\cite{SungJin_Ill} in which they prove ill-posedness in standard Sobolev spaces of a related quasilinear Schr\"odinger equation appearing in~\cite{MR1387617,MR3994381}, though the techniques apply as well to show the ill-posedness of \eqref{QLS} in standard Sobolev spaces. This result uses ideas from their previous result~\cite{2019arXiv190202025J}, in which they are able to exploit the aforementioned heuristic to rigorously prove ill-posedness of degenerate solutions of the Hall-MHD and electron-MHD equations. The same authors are also working to develop an alternative approach to local well-posedess using function spaces suited to the degeneracy of the initial data~\cite{SungJin_Well}. Ill-posedness of a related degenerate model was also considered in~\cite{MR2967120}. 

For sublinear endpoint decay one may obtain local well-posedness using polynomially weighted spaces as in~\cite{MR4011864}. Thus, our main concern in this article will be the problem of linear endpoint decay: the sharp decay rate we expect to be well-posed in any reasonable sense. Our motivation for considering data satisfying conditions of this type is due to the following result from~\cite{MR4042218}:
\begin{thm}{\bf\cite{MR4042218}}\label{thrm:GHGM}
If \(\mu = 1\) there exists a unique (up to translation) non-negative minimizer \(\phi = \phi_\omega\) of the Hamiltonian \(H[\phi]\) for fixed mass \(M[\phi] = \sqrt{2\pi}\omega >0\) given by
\begin{equation}\label{compacton}
\phi_\omega(x) := \sqrt{2\omega}\cos\left(\tfrac x{\sqrt 2}\right) \bbo_I(x),
\end{equation}
where the interval \(I = (-\frac\pi{\sqrt 2},\frac\pi{\sqrt 2})\).
\end{thm}

One may construct compact breather\footnote{We adopt the terminology of~\cite[Section~9.2.2]{MR3829411} and refer to these solutions as ``breathers''. One might also refer to these solutions as ``standing waves''.} solutions of \eqref{QLS} from the minimizer \(\phi_\omega\) for any \(\theta\in \R\) by
\begin{equation}\label{compactbreather}
u(t,x) = e^{-it\omega + i\theta}\phi_\omega(x).
\end{equation}
Compact breathers and compactons (the corresponding analogue for KdV-type equations) are an important feature of (focusing) degenerate dispersive equations. There has been a substantial volume of work on the existence and properties of compact solutions of this type for a variety of problems, in particular the work of Rosenau and collaborators, e.g.~\cite{PhysRevE.89.022924,PhysRevLett.94.045503,MR3402749,ROSENAU200644,MR2159688,MR2133462,MR2601818,HR,MR1294558,PhysRevLett94174102,2017arXiv170903322Z,MR3628993,17518121499095101,PIKOVSKY200656}. We refer the reader to the primer~\cite{MR3829411} for a review of the current state of the art.  

The main result of this article is local well-posedness of the equation \eqref{QLS} in a space that contains the solution \eqref{compactbreather}:
\begin{thm}[Local well-posedness of \eqref{QLS}]\label{thrm:Main}
Let \(x_0>0\) and \(I=(-x_0,x_0)\). Then there exists a set \(S\subset L^2\) of functions that are non-zero and smooth on \(I\), supported on \(\bar I\) and satisfy \eqref{DecayRate} so that for any \(u_0\in S\) there exists a time \(T>0\) and a unique \(u\in \Cont([0,T];L^2)\) that satisfies \eqref{QLS} in the sense of distributions with \(u\in \Cont([0,T]\times \R)\) and \(u_x,(\frac12 u^2)_{xx}\in \Cont([0,T];L^\infty)\). For all \(t\in[0,T]\) the solution \(u(t)\) is non-zero and smooth on \(I\), supported on \(\bar I\), and conserves its mass, momentum and energy. Further, for all \(t\in[0,T]\) the solution map \(u_0\mapsto u(t)\) is Lipschitz continuous with respect to the \(L^2\)-topology.
\end{thm}

As far as we are aware, this is the first proof of local well-posedness for a degenerate quasilinear Schr\"odinger equation. A key innovation in this article is that, unlike in the case of the KdV equation considered in~\cite{MR4011864}, we are able to handle the critical (linear) endpoint decay rate.

Critical endpoint decay rates were previously considered in the setting of the shoreline problem for a model of shallow water waves in~\cite{MR3817287}. In this case, the finite speed of propagation allows the authors to work with a finite number of (weighted) derivatives, as in the subcritical endpoint decay rates considered in~\cite{MR4011864}. In the Schr\"odinger case, where the speed of propagation is infinite, we are no longer able to work with a finite number of weighted derivatives, which significantly complicates the analysis.

The set \(S\), which will be described in detail Section~\ref{sect:Coords}, essentially consists of solutions that are analytic with respect to the weighted derivative \(|u_0|\partial_x\). This set is extremely ``small'' in any reasonable sense (for example, every element of \(S\) must be analytic on the open interval \(I\)) and is certainly far from optimal in the case of sublinear endpoint decay. However, the set \(S\) does contain the compact breather \eqref{compactbreather} and reasonable perturbations thereof (see Proposition~\ref{pertprop}). This motivates us to consider its stability. As a corollary to Theorems~\ref{thrm:GHGM},~\ref{thrm:Main}, we may apply the method of~\cite{cazenave1982orbital} to obtain the following stability result, which we prove in Section~\ref{sect:STABILITY}:
\begin{thm}[Stability of the energy-minimizing breather]\label{thrm:OrbitalStability}
Let \(\mu=1\), \(\omega>0\) and \(\epsilon>0\). Then there exists some \(\delta>0\) so that for any \(u_0\in S\) satisfying
\begin{equation}\label{StabilityHypothesis}
\inf_{\theta,h\in \R}\|u_0(\cdot + h)^2 - e^{2i\theta}\phi_\omega^2\|_{L^1\cap\dot H^1}\leq \delta,
\end{equation}
we have the estimate
\begin{equation}\label{StabilityConcolusion}
\sup\limits_{t\in[0,T]}\left(\inf\limits_{\theta,h\in \R}\|u(t,\cdot+h)^2 - e^{-2it\omega + 2i\theta}\phi_\omega^2\|_{L^1\cap\dot H^1}\right)\leq\epsilon,
\end{equation}
where \(T>0\) is the lifespan of the solution \(u\) obtained in Theorem~\ref{thrm:Main}.
\end{thm}

\begin{rem}
Somewhat unusually, the stability in Theorem~\ref{thrm:OrbitalStability} is obtained in terms of \(u^2\) instead of \(u\). This topology arises naturally from the observation that, for \(q := u^2\) and \(\mu = 1\), the mass and energy may be written as
\[
M = \int |q|\,dx\qtq{and} H = \tfrac14\int|q_x|^2\,dx - \tfrac12\int |q|^2\,dx.
\]
\end{rem}

\begin{rem}
The proof of Theorem~\ref{thrm:OrbitalStability} applies to any interval \([0,T]\) on which the solution \(u\) of \eqref{QLS} conserves both the mass and energy, and \(u^2\in C([0,T];L^1\cap \dot H^1)\). Theorem~\ref{thrm:Main} guarantees that this timescale is at least non-trivial, but it is currently far from clear whether or not one expects to be able to take arbitrarily large \(T\). Indeed, in the corresponding KdV case, it was been proved in~\cite{2017arXiv170903322Z,MR4011864} that the support of the solution cannot remain constant on arbitrarily long timescales, which suggests a possible finite time breakdown of regularity. Whether or not such a phenomenon holds for \eqref{QLS} and whether or not this can violate orbital stability on sufficiently long timescales remains an interesting open problem.
\end{rem}

\begin{rem}
In \cite{MR4042218}, the authors also construct traveling compactons as solutions to \eqref{QLS}, however these states arise at the expense of a highly singular phase and hence significantly complicate the regularity and boundary conditions considered.  As a result, at present we leave stability of these states as an open problem.
\end{rem}

\subsection*{Outline of the proof} In our previous work~\cite{MR4011864} on the KdV analogue of \eqref{QLS}, we use the hydrodynamic formulation of the problem (see \eqref{Hydrodynamic} below) to switch to Lagrangian coordinates, which has the effect of freezing the degeneracy at the initial time \(t = 0\). We then make a change of independent variable to flatten the degeneracy and reduce the problem to a non-degenerate quasilinear equation, which can be solved using the energy method. These changes of variable were inspired by similar approaches in related degenerate problems, e.g.~\cite{MR1378250,KochHT,MR3375536,MR3375546,MR3197240,MR3544328,MR1668860,MR2839301,MR2779087,MR2547977,MR3817287}.

While a similar approach formed the basis of our original investigation of \eqref{QLS}, a key difficulty was encountered due to the need to work in spaces of analytic functions. If the initial data has linear endpoint decay, after changing the independent variable, we are required to propagate exponential decay of the initial data to the solution. Even in the case of a constant coefficient linear Schr\"odinger equation, propagation of exponential decay of the data requires controlling the solution in spaces of analytic functions. However, as the approach of~\cite{MR4011864} requires working with a quasilinear Schr\"odinger equation, one must bound the solution in spaces of exponentially weighted analytic data adapted to a variable metric.

To circumvent this difficulty, we introduce two key new ideas in this article. The first is a change of independent variable that prioritizes flattening the degeneracy and reduces \eqref{QLS} to a derivative semilinear Schr\"odinger equation. This significantly simplifies the problem of controlling our solution in spaces of analytic functions. The second is to couple \eqref{QLS} with an equation for \(w:=\frac{\bar u u_x}{|u|}\), which controls the decay of the solution \(u\). Indeed, by controlling this function pointwise we will be able to work in translation-invariant (with respect to the new independent variable) spaces. This not only simplifies the nonlinear estimates considerably, but also allows us to replace the asymptotic assumption \eqref{endpointasymptotics}, with the far less prescriptive assumption \eqref{DecayRate}.

Our motivation for considering the function \(w\) is most readily understood by writing the equation \eqref{QLS} in the form
\[
iu_t = (|u|\partial_x)^2u + i (\Im w) (|u|\partial_x)u + \mu |u|^2u.
\]
After making our change of independent variable, which maps \(|u|\partial_x\mapsto \partial_y\), controlling the sub-principle term requires controlling \(\Im w\). To do this, we must consider the complex-valued function \(w\) rather than just its imaginary part. Indeed, we may compute that \(w\) satisfies the Schr\"odinger equation,
\[
iw_t = (|u|\partial_x)^2w + \text{lower order terms}.
\]
The variable \(w\) also arises naturally from the hydrodynamic formulation of \eqref{QLS}
\begin{equation}\label{Hydrodynamic}
\begin{cases}
\rho_t + (v\rho)_x = 0,\smallskip\\
v_t + vv_x + (v^2 - \rho\rho_{xx} + \tfrac12\rho_x^2 - \mu \rho^2)_x = 0,
\end{cases}
\end{equation}
where \(\rho,v\) are defined as in \eqref{HydroVariables}. We may then compute that
\[
w = \frac{\rho_x}{2\sqrt{\rho}} - i\frac{v}{2\sqrt{\rho}}.
\]

Unfortunately, at least in the case of linear endpoint decay, the semilinear equation we obtain for \(w\) in our new coordinate system fails the Takeuchi-Mizohata condition~\cite{MR860041,MR564671} for well-posedness of linear Schr\"odinger equations in Sobolev spaces (see also~\cite{MR759481,MR3266990,MR2967120,MR3179690}). To address this issue, as we have already alluded to, we work in spaces of analytic functions. By allowing the radius of analyticity to shrink linearly in time, we obtain a global smoothing effect that is sufficient to control the problematic terms. We remark that similar analytic spaces and estimates have a long history of application to PDEs (and even ODEs) and have previously appeared in several works on Schr\"odinger equations, e.g.~\cite{MR2265624,MR1046275,MR1054532,MR2063545,MR870865,MR1360541,MR1152001,MR1791328,MR3917735,MR1855975,MR3669837}.

Another difficulty we encounter with our semilinear equations for \(u,w\) is a transport term with unbounded velocity. This prevents us from using a contraction mapping argument, so instead our proof of existence relies on an energy method: we construct solutions as weak limits of regularized equations. We remark that the fact that we are unable to use a contraction mapping argument is unsurprising, given that the original equation \eqref{QLS} is quasilinear. This unbounded velocity term also prevents us from comparing two solutions in our new coordinate system. To prove uniqueness and continuity we instead use the original equation \eqref{QLS}.

While the function \(w\) significantly simplifies some of the analysis, it has the disadvantage that we do not expect it will decay at spatial infinity in our new coordinates (at least in the case of linear endpoint decay). To handle this, we bound low frequencies in \(L^\infty\) and high frequencies in Sobolev spaces. This enables us to treat non-decaying data, while still using energy estimates to control the high frequency contributions.

The remainder of the paper is structured as follows: In Section~\ref{sect:Coords} we discuss the change of variables and define the set \(S\) of initial data. We provide additional details for some of the more involved computations appearing in this section in Appendix~\ref{app:Computations}. In Section~\ref{sect:Prelims} we prove several preliminary estimates for our spaces of analytic functions. Our main nonlinear estimates are stated in Proposition~\ref{prop:NonlinearEstimates} but, as they are standard albeit technical, we delay the proof until Appendix~\ref{app:Multilinear}.

We begin our proof of existence of solutions to \eqref{QLS} with a priori estimates for model linear equations in Section~\ref{sect:Linear}. We then apply these in Section~\ref{sect:Existence} to obtain a solution of the semilinear equations described above. Once we have solved the semilinear problems to obtain \(u,w\) in our new coordinate system, it remains to verify that the solution we construct has sufficient regularity to invert the change of coordinates and obtain a solution to the original equation \eqref{QLS}. This is the main task in Section~\ref{sect:Invert}, where we complete the proof of Theorem~\ref{thrm:Main}.

Finally, in Section~\ref{sect:STABILITY} we prove our stability result, Theorem~\ref{thrm:OrbitalStability}

\subsection*{Acknowledgement} This project was started as a collaboration with Pierre Germain, to whom the authors are extremely grateful for many fruitful discussions and several invaluable contributions towards understanding the structure of the problem. The authors also wish to thank Sung-Jin Oh for several enlightening discussions about degenerate dispersive equations. Finally, the authors would like to thank the anonymous referee for their careful reading of the manuscript and several insightful comments and suggestions.

\section{Reformulation of the problem}\label{sect:Coords}

In this section we introduce the various changes of variable required to prove Theorem~\ref{thrm:Main} and define the set \(S\) of initial data.

\subsection{Changes of variable}
Motivated by the linearization of \eqref{QLS}, we introduce the independent variable
\begin{equation}\label{Coords}
y(t,x) = \int_0^x \frac1{|u(t,\zeta)|}\,d\zeta + c(t)\qtq{for}x\in I,
\end{equation}
where \(c(t)\) is a real-valued, continuously differentiable function satisfying \(c(0) = 0\). As we are assuming \(u_0\) is non-zero on \(I\) as well as \eqref{DecayRate}, the map \(x\mapsto y(0,x)\) is readily seen to be a diffeomorphism from \(I\) onto \(\R\). The integrand is designed precisely to flatten the degeneracy, whereas the time-dependent constant \(c(t)\) will be chosen shortly to provide a helpful cancellation. The freedom to choose \(c(t)\) is due to the gauge-invariance of this change of variables: we have total freedom to decide the value of \(y(t,0)\).

Using this change of variable we define
\begin{equation}\label{U-W-def}
U(t,y(t,x)) := u(t,x)\qtq{and} W(t,y(t,x)) := w(t,x) = \frac{\bar u(t,x)u_x(t,x)}{|u(t,x)|}.
\end{equation}
The variable \(W\) will be used to control the decay of \(U\) and is related to \(U\) by the identity
\begin{equation}\label{UtoW}
W = \frac{\bar U U_y}{|U|^2}.
\end{equation}

We denote the real and imaginary parts of \(W\) by
\begin{equation}\label{alpha-beta-def}
\alpha := \Re W\qtq{and} \beta := \Im W,
\end{equation}
and use the functions \(\alpha,\beta\) to fix our gauge by taking \(c(t)\) to solve the equation
\begin{equation}\label{GaugeFixing}
\begin{cases}
c_t(t) = - \beta(t,c(t)) - 3\displaystyle\int_0^{c(t)}\alpha(t,\zeta)\beta(t,\zeta)\,d\zeta,\smallskip\\
c(0) = 0 ,
\end{cases}
\end{equation}
so that the equation \eqref{QLS} can be written as
\begin{equation}\label{U-eqn}
i\left(U_t + bU_y\right) = U_{yy} + 2i\beta U_y + \mu |U|^2U,
\end{equation}
where the real-valued coefficient
\begin{equation}\label{b-def}
b(t,y) := - 3\int_0^y\alpha(t,\zeta)\beta(t,\zeta)\,d\zeta
\end{equation}
satisfies
\[
y_t(t,x) = b(t,y(t,x)) - \beta(t,y(t,x)).
\]
As discussed in the introduction, in order to solve the equation \eqref{U-eqn} we will also need to control \(W\), which we compute satisfies the equation
\begin{equation}\label{W-eqn}
i\left(W_t + bW_y\right) = W_{yy} + (2W^2 - \tfrac12|W|^2)_y + 3i\alpha\beta W + 2\mu |U|^2\alpha.
\end{equation}

For the reader's convenience, we outline these computations in Appendix~\ref{app:Computations}.

We conclude our discussion of the change of variables by performing these computations in the special case of the compact breather:
\begin{ex}[The compact breather in \(y\)-coordinates]\label{ex:CB}
Let \(\theta\in \R\) and \(u(t,x) = e^{-it\omega + i\theta}\phi_\omega(x)\) be as in \eqref{compactbreather}. Then, for \(x\in I = (-\frac\pi{\sqrt 2},\frac\pi{\sqrt 2})\) we have
\[
y(t,x) = \int_0^x \frac1{\sqrt{2\omega}\cos(\frac \zeta{\sqrt2})}\,d\zeta = \tfrac1{\sqrt\omega}\ln\left(\tan\left(\tfrac x{\sqrt2}\right) + \sec\left(\tfrac x{\sqrt 2}\right)\right),
\]
where we note that \(\Im w = 0\) and hence \(c\equiv 0\). As a consequence, we have
\begin{align*}
U(t,y) &= e^{-it\omega + i\theta}\sqrt{2\omega}\sech(\sqrt\omega y),\\
W(t,y) &= - \sqrt\omega\tanh(\sqrt\omega y).
\end{align*}
\end{ex}

\subsection{Function spaces}
It is natural to bound the solution \(U\) in Sobolev-type spaces. Given \(s\geq0\) we define the Sobolev space \(H^s\) with norm
\[
\|f\|_{H^s}^2 := \int \<\xi\>^{2s}| \hat{f} (\xi)|^2\,d\xi,
\]
where \(\<\xi\> = \sqrt{1 + |\xi|^2}\), and the Fourier transform
\[
\hat{f} (\xi) := \tfrac1{\sqrt{2\pi}}\int f(x)e^{-ix\xi}\,dx.
\]

In the case of linear endpoint decay, it is clear from Example~\ref{ex:CB} that we should not expect \(W\) to decay as \(|y|\to\infty\). This motivates us to introduce the space \(Z^s\) with norm
\[
\|f\|_{Z^s} := \|f\|_{L^\infty} + \|f_y\|_{H^{s-\frac12}}.
\]
The space \(Z^s\) is often referred to as the Zhidkov space due to its original appearance in~\cite{MR1831831,MR906067} and has been applied to study the NLS with non-vanishing boundary conditions, e.g.~\cite{MR2099970,MR2317389}.

In order to control the subprinciple terms in the equations for \(U\) and \(W\) we will need our solution to be analytic. Given a function \(m\colon \R\rightarrow \C\) we define the Fourier multiplier
\[
m(D_y)f (x) := \tfrac1{\sqrt{2\pi}}\int m(\xi) \hat f(\xi)e^{ix\xi}\,d\xi.
\]
Given \(\tau>0\) and a Banach space \(X\) of tempered distributions on \(\R\) with norm \(\|\cdot\|_X\) we define the subspace \(AX_\tau\) of \(X\) to consist of \(f\in X\) with finite norm
\[
\|f\|_{AX_\tau} := \|(e^{\tau D_y} f,e^{-\tau D_y}f)\|_X,
\]
where, for concreteness, we make the convention that if \(g = (g_1,g_2,\dots,g_n)\) then
\[
\|g\|_X = \sum\limits_{j=1}^n\|g_j\|_X.
\]
In particular, the space \(AH^s_\tau\) coincides with the definition of the analytic Gevrey spaces appearing in~\cite{MR1026858,MR870865}.

Before turning to the definition of the set \(S\) of initial data and stating the existence part of Theorem~\ref{thrm:Main} in \(U,W\) coordinates, it will be useful to introduce a little more notation.

Given \(T>0\) and a Banach space \(X\) of tempered distributions we define the space \(\Cont([0,T];X)\) to consist of continuous functions \(f\colon [0,T]\rightarrow X\) and be endowed with the supremum norm. Due to the presence of a complex transport term in the equation for \(W\), we will need the radius of analyticity \(\tau\) to be time-dependent. As a consequence, given \(T>0\) and a continuous function \(\tau\colon[0,T]\to (0,\infty)\) we say that \(f\in \Cont([0,T];AX_{\tau})\) if the vector-valued function \((e^{\tau D_y} f,e^{-\tau D_y}f)\in \Cont([0,T];X)\).

In a similar fashion, we say that \(f\in L^p((0,T);AX_{\tau})\) if \((e^{\tau D_y} f,e^{-\tau D_y}f)\in L^p((0,T);X)\) and denote
\[
\|f\|_{L^p_TX}^p = \int_0^T \|f(t)\|_X^p\,dt,\qquad \|f\|_{L^p_TAX_\tau}^p = \int_0^T\|f(t)\|_{AX_\tau}^p\,dt,
\]
with the obvious modification when \(p = \infty\).

Finally, if \(X\) has a predual then we write \(f\in \Cont_w([0,T];AX_{\tau})\) if \(f\in L^\infty((0,T);AX_{\tau})\) and \((e^{\tau D_y} f,e^{-\tau D_y}f)(t) \overset\ast\rightharpoonup (e^{\tau D_y} f,e^{-\tau D_y}f)(s)\) in \(X\) as \(t\to s\) for \(t,s\in [0,T]\).

\subsection{Existence for \(U,W\)}~
In order to prove existence for \eqref{QLS}, we will prove existence of a solution to the equation \eqref{U-eqn}. Here it will be useful to treat the equations \eqref{U-eqn} and \eqref{W-eqn} as a system, where the initial data is not necessarily related via the identity \eqref{UtoW}. We then have the following theorem:

\begin{thm}\label{thrm:UW-existence}
Let \(0<s\leq\frac12\) and \(0<\tau_0\leq 1\). Then, given any
\[
U_0\in AH^s_{\tau_0}\quad\text{and}\quad W_0\in AZ_{\tau_0}^s,
\]
there exists some \(T>0\), a non-increasing, continuously differentiable function \(\tau\colon[0,T]\to(0,\infty)\) so that \(\tau(0) = \tau_0\), and a solution \(U\in \Cont_w([0,T];AH^s_\tau)\), \(W\in \Cont_w([0,T];AZ_{\tau}^s)\) of the system \eqref{U-eqn}, \eqref{W-eqn} with initial data \(U(0) = U_0\) and \(W(0) = W_0\).

Further, we have the estimates
\begin{align}
\|U\|_{L^\infty_TAH^s_\tau}&\lesssim \|U_0\|_{AH^s_{\tau_0}},\label{U-AP}\\
\|W\|_{L^\infty_TAZ^s_\tau}&\lesssim \|U_0\|_{AH^s_{\tau_0}} + \|W_0\|_{AZ^s_{\tau_0}},\label{W-AP}
\end{align}
and the lower bound
\begin{equation}\label{T-bound}
T\gtrsim \frac1{\|U_0\|_{AH^s_{\tau_0}}^2 + \|W_0\|_{AZ^s_{\tau_0}}^2}.
\end{equation}
\end{thm}

\begin{rem}
The solutions we construct will be obtained by taking a weak limit of a regularized system of equations. Thus, a priori, our solution is a distributional solution of \eqref{U-eqn}, \eqref{W-eqn}. However, for any \(n\geq 0\) the space of bounded \(\Cont^n\) functions is (locally compactly) embedded in both \(AH^s_\tau\) and \(AZ^s_\tau\), so the corresponding \(U,W\) are smooth classical solutions of the equations \eqref{U-eqn}, \eqref{W-eqn}.
\end{rem}

\begin{rem}
The assumptions that \(0<s\leq \frac12\) and \(0<\tau_0\leq 1\) are solely for technical convenience and can be replaced by \(s>0\) and \(\tau_0>0\) by making suitable modifications to the various estimates.
\end{rem}

\subsection{The initial data set \(S\)} From the statement of Theorem~\ref{thrm:UW-existence} we obtain the following definition for the data set \(S\):
\begin{df}[The data set \(S\)]
Let \(S\) be the set of \(u_0\in L^2\) that are non-zero and smooth on \(I\), supported on \(\bar I\), and satisfy \eqref{DecayRate}, so that if we define
\[
y_0(x) = \int_0^x \frac1{|u_0(\zeta)|}\,d\zeta\quad\text{and}\quad U_0(y_0(x)) := u_0(x),
\]
then there exists some \(0<s\leq \frac12\) and \(0<\tau_0\leq 1\) so that
\[
U_0\in AH^s_{\tau_0}\quad \text{and}\quad\frac{\bar U_0 U_{0y}}{|U_0|^2}\in AZ^s_{\tau_0}.
\]
\end{df}

Due to the implicit nature of the definition of \(S\), it is not immediately clear what a typical element looks like. Example~\ref{ex:CB} shows that for any \(\omega>0\) the compact breather \(\phi_\omega\in S\). Further, given \(|\epsilon|<1\) we have \((1 + \epsilon)\phi_\omega = \phi_{(1 + \epsilon)^2\omega}\in S\). To obtain a slightly larger class of examples to which Theorem~\ref{thrm:OrbitalStability} may be applied, we conclude this section with an explicit construction of an admissible perturbation of the compact breather solution.

\begin{prop}[An admissible perturbation of the compact breather]
\label{pertprop}
Let \(\mu=1\), \(\omega > 0\) and the interval \(I = (-\frac\pi{\sqrt 2},\frac\pi{\sqrt 2})\). Let \(M,C\geq 1\) and \(f\colon I\rightarrow \R\) be a smooth function so that for any \(n\geq 0\) we have
\begin{equation}\label{BoundedDerivatives}
|\partial_x^n f(x)|\leq MC^n.
\end{equation}
Then,
\begin{equation}
\label{pert:ansatz}
    u_0 = e^{if}\phi_\omega\in S.
\end{equation}
\end{prop}
\begin{proof}
First, we observe that if there exist constants \(K,B>0\) so that for all \(n\geq 0\) we have
\begin{equation}\label{Mercury}
\tfrac 1{n!}\|\partial_y^n f\|_{L^2}\leq K B^n,
\end{equation}
then taking \(0 < \tau_0< \frac 1B\) we have \(f\in AH^s_{\tau_0}\) for any \(s\in \R\). Similarly, if 
\begin{equation}\label{Venus}
\tfrac 1{n!}\|\partial_y^n f\|_{L^\infty}\leq K B^n,
\end{equation}
then taking \(0 < \tau_0< \frac 1B\) we have \(f\in AL^\infty_{\tau_0}\).

Second, we observe that if we have the pointwise bound
\[
\tfrac1{n!}|\partial_y^n f_j|\leq K_j B_j^n,
\]
for constants \(B_j>0\) and functions \(K_j = K_j(y)>0\), then
\begin{align}
\tfrac1{n!}|\partial_y^n(f_1f_2)| &\leq (n+1)K_1K_2 \max\{B_1^n,B_2^n\},\label{analytic-prod-1}\\
\tfrac1{n!}|\partial_y^n(f_1f_2f_3)|&\leq \tfrac{(n+2)(n+1)}2K_1K_2K_3 \max\{B_1^n,B_2^n,B_3^n\}.\label{analytic-prod-2}
\end{align}
By induction on \(n\), using that \(\sech''(y) = \sech y - 3\sech^3 y\) and the inequality \eqref{analytic-prod-2},  we may then bound
\begin{equation}\label{sech-bound}
\tfrac1{n!}|\partial_y^n\sech y|\leq 2^n\sech y,
\end{equation}
and similarly, using that \(\tanh''(y) = 2\tanh^3 y - 2\tanh y\),
\[
\tfrac1{n!}|\partial_y^n\tanh y|\leq 2^n.
\]

Next, we compute that
\[
W_0(y) = -\sqrt\omega\tanh(\sqrt\omega y) + iF(y)\sqrt{2\omega}\sech(\sqrt\omega y),
\]
where
\[
F(y) = f'\left(\sqrt2\arctan\left(\sinh\left(\sqrt\omega y\right)\right)\right).
\]
Using that
\[
\frac d{dy}\left(\sqrt2\arctan\left(\sinh\left(\sqrt\omega y\right)\right)\right)= |U_0(y)| = \sqrt{2\omega}\sech(\sqrt \omega y),
\]
we may apply the Fa\`a di Bruno formula to obtain
\[
\partial_y^n F = \sum\limits_{k=1}^n \partial_x^{k+1} f\left(\sqrt2\arctan\left(\sinh\left(\sqrt\omega y\right)\right)\right)\cdot B_{n,k}\left(|U_0|,\partial_y|U_0|,\dots,\partial_y^{n-k}|U_0|\right),
\]
where \(B_{n,k}\) is the partial Bell polynomial. Using the estimate \eqref{sech-bound}, the hypothesis \eqref{BoundedDerivatives}, and properties of the Bell polynomials (see e.g.~\cite{MR0460128}), we may bound
\[
\tfrac1{n!}|\partial_y^n F|\leq MCB^n\qtq{where}B = 2\sqrt\omega \max\left\{C,\sqrt{2\omega}\right\}.
\]

Applying \eqref{Mercury}, \eqref{Venus}, and \eqref{analytic-prod-1}, we may then choose \(0<\tau_0\ll1\) sufficiently small to ensure that \(W_0\in AZ^s_\tau\) for any \(s\in \R\). Finally, we observe that \(U_{0y} = U_0W_0\) and hence we may apply \eqref{Mercury} and \eqref{analytic-prod-1} to conclude that \(U_0\in AH^s_{\tau_0}\) for any \(s\in \R\).
\end{proof}

\begin{rem}
We may use the numerical methods presented in \cite{MR4042218} to explore perturbations of the compacton of the form in Proposition \ref{pertprop}.  In Figure \ref{fig:pertfig}, we present a time sequence of plots for solutions to \eqref{QLS} of the form \eqref{pert:ansatz} demonstrating that numerical stability is observed in a reasonable fashion on short time scales as the perturbation spreads towards the endpoints.
\end{rem}

\begin{figure} 
\centering
\includegraphics[width=.25\textwidth]{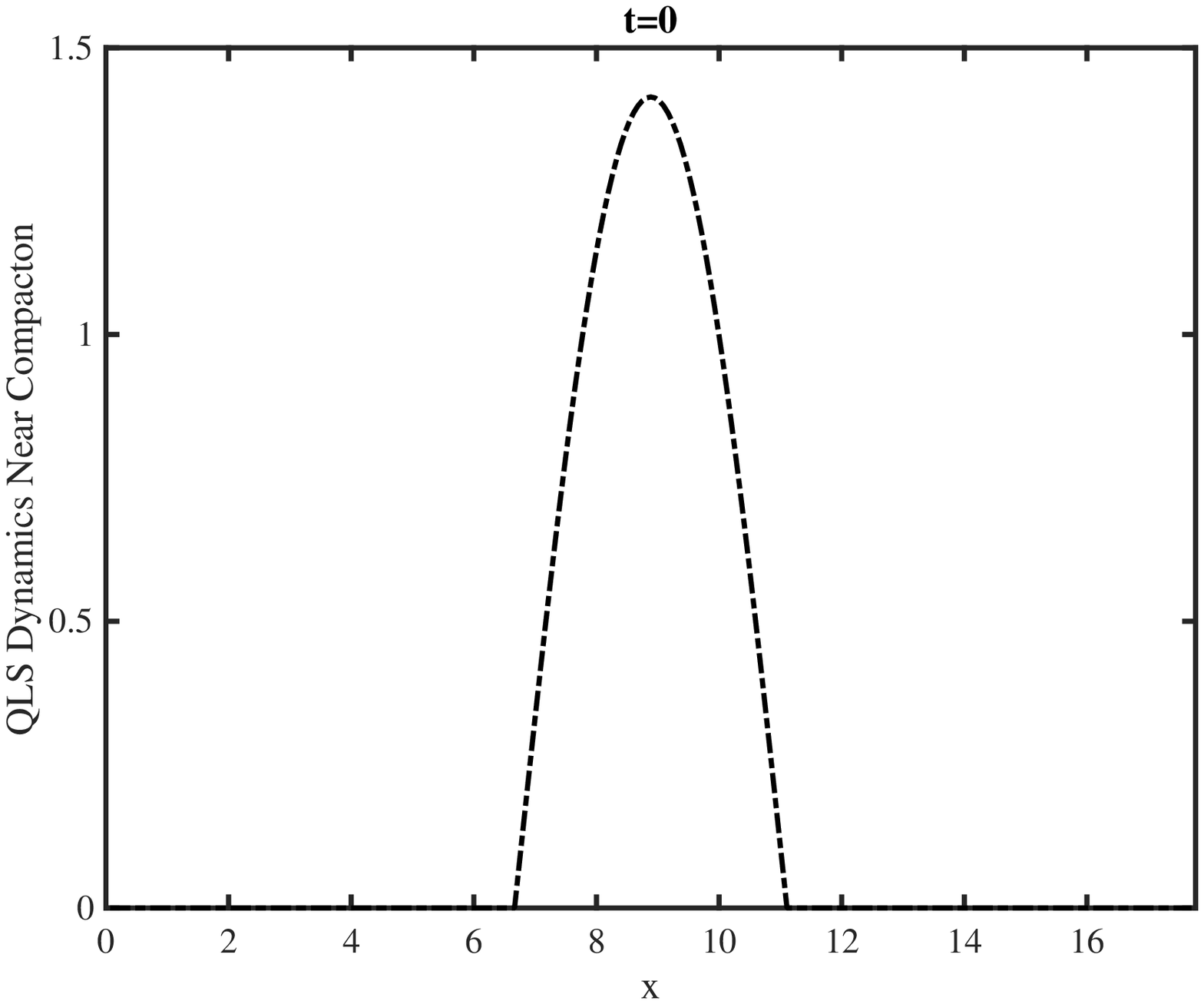}
\includegraphics[width=.25\textwidth]{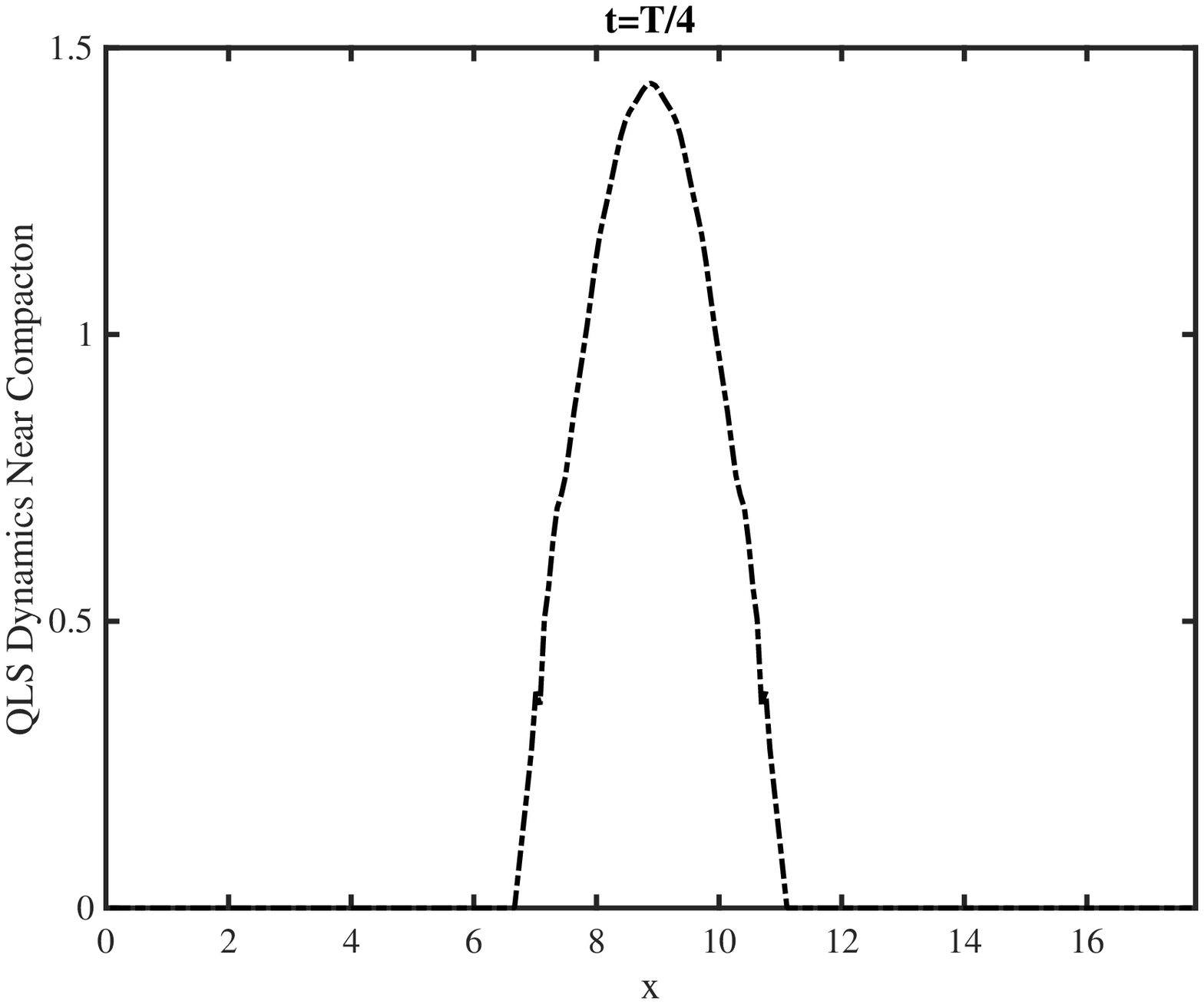}
\includegraphics[width=.25\textwidth]{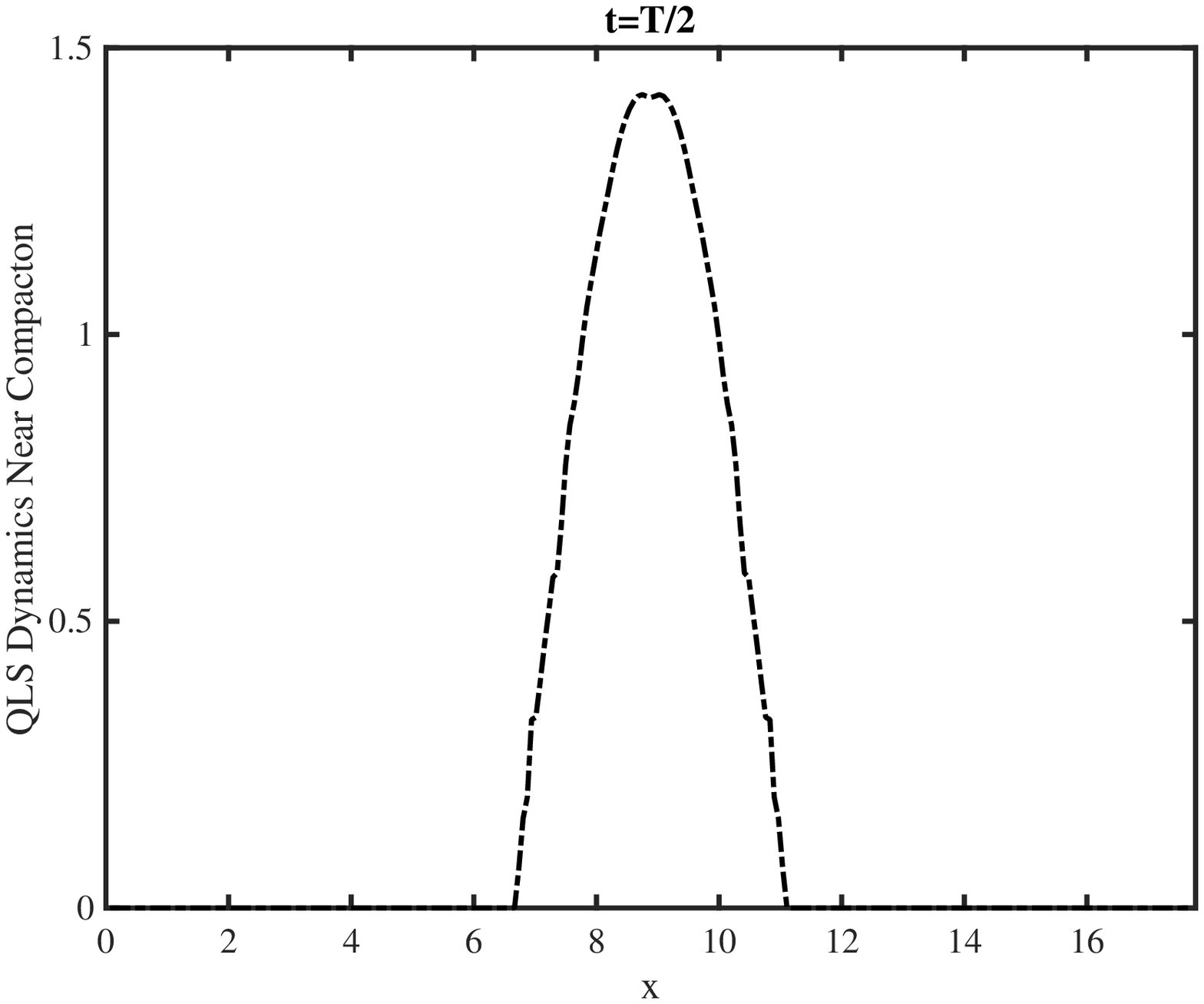} \\
\includegraphics[width=.25\textwidth]{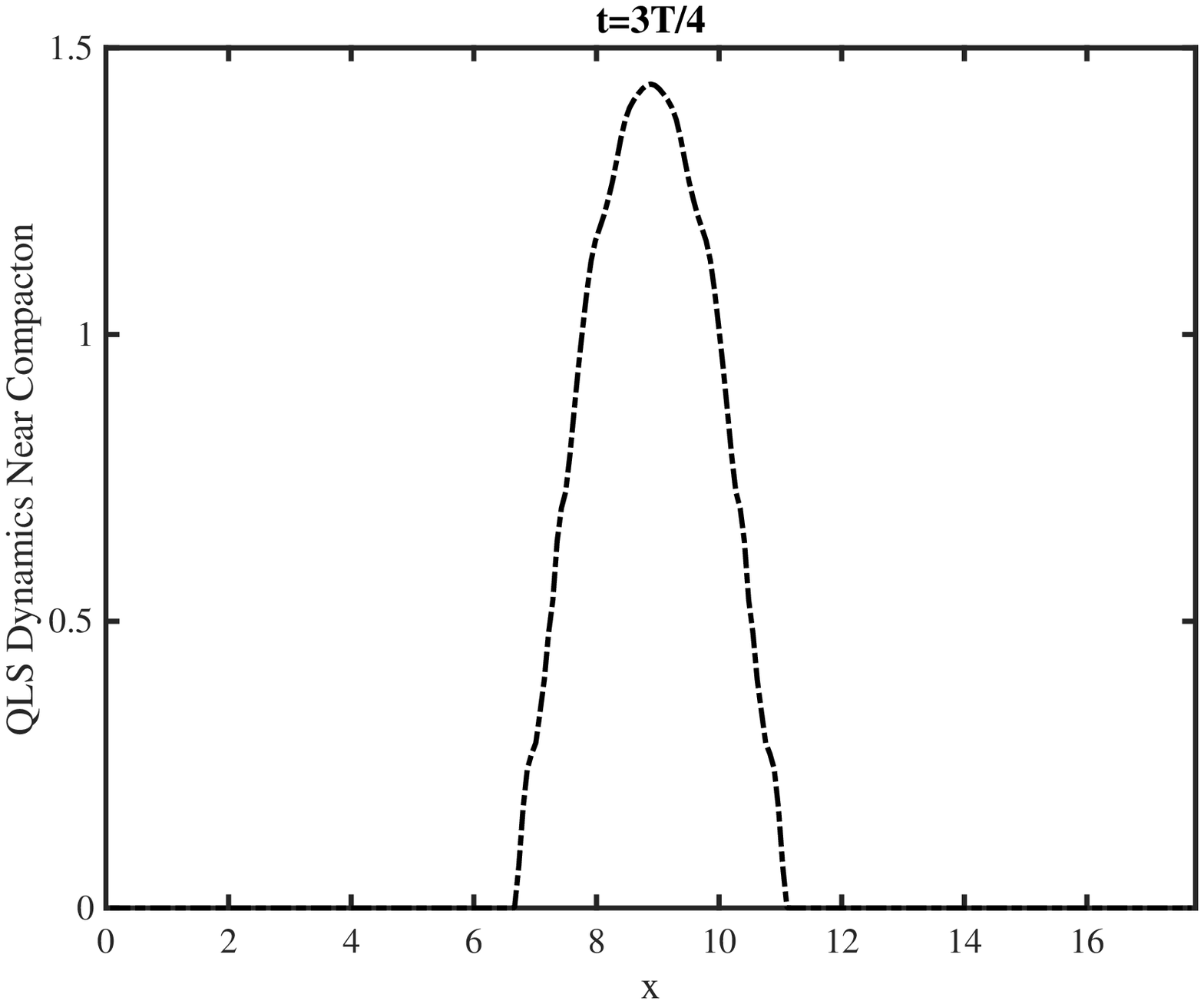} 
\includegraphics[width=.25\textwidth]{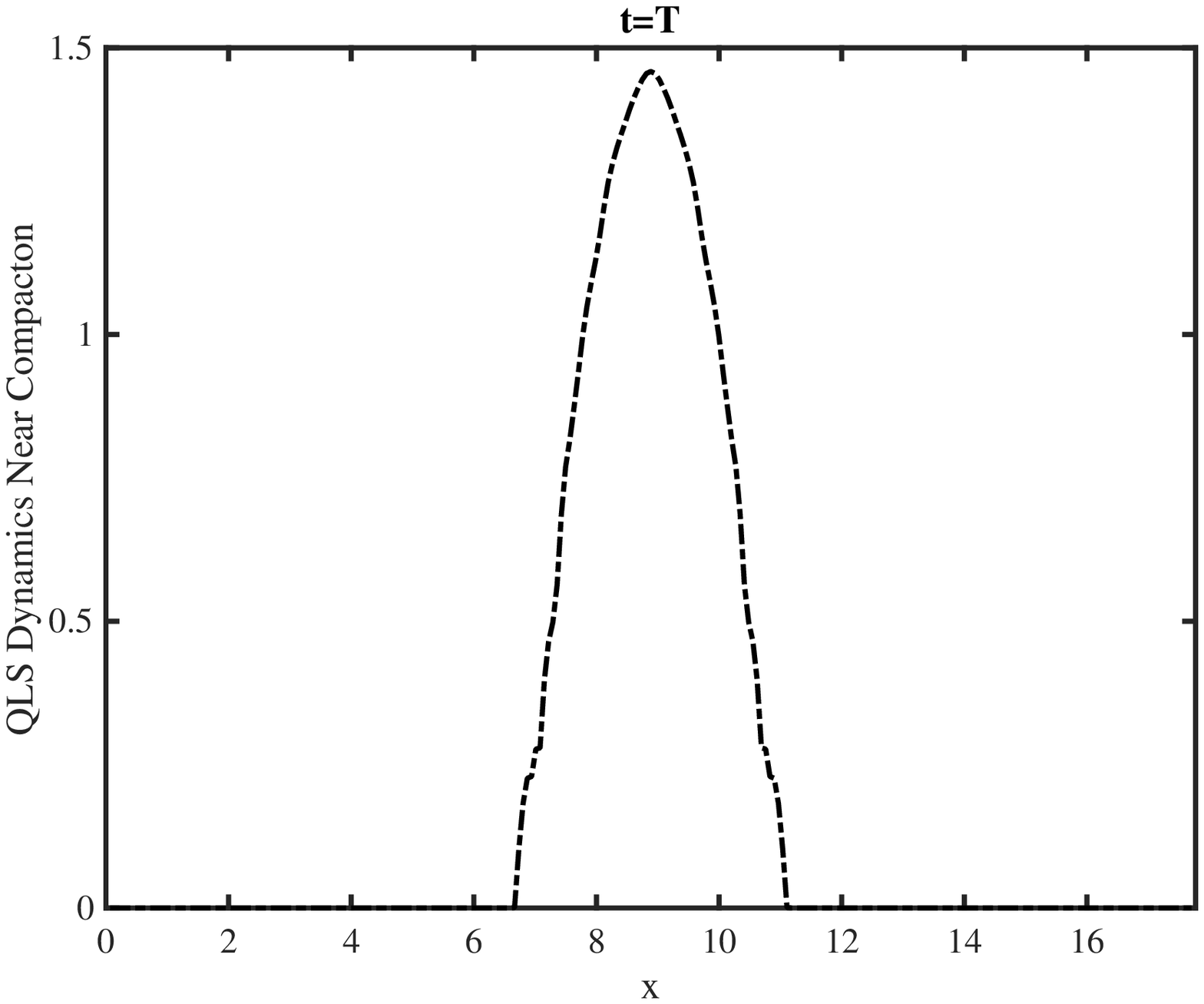} \\
\includegraphics[width=.4\textwidth]{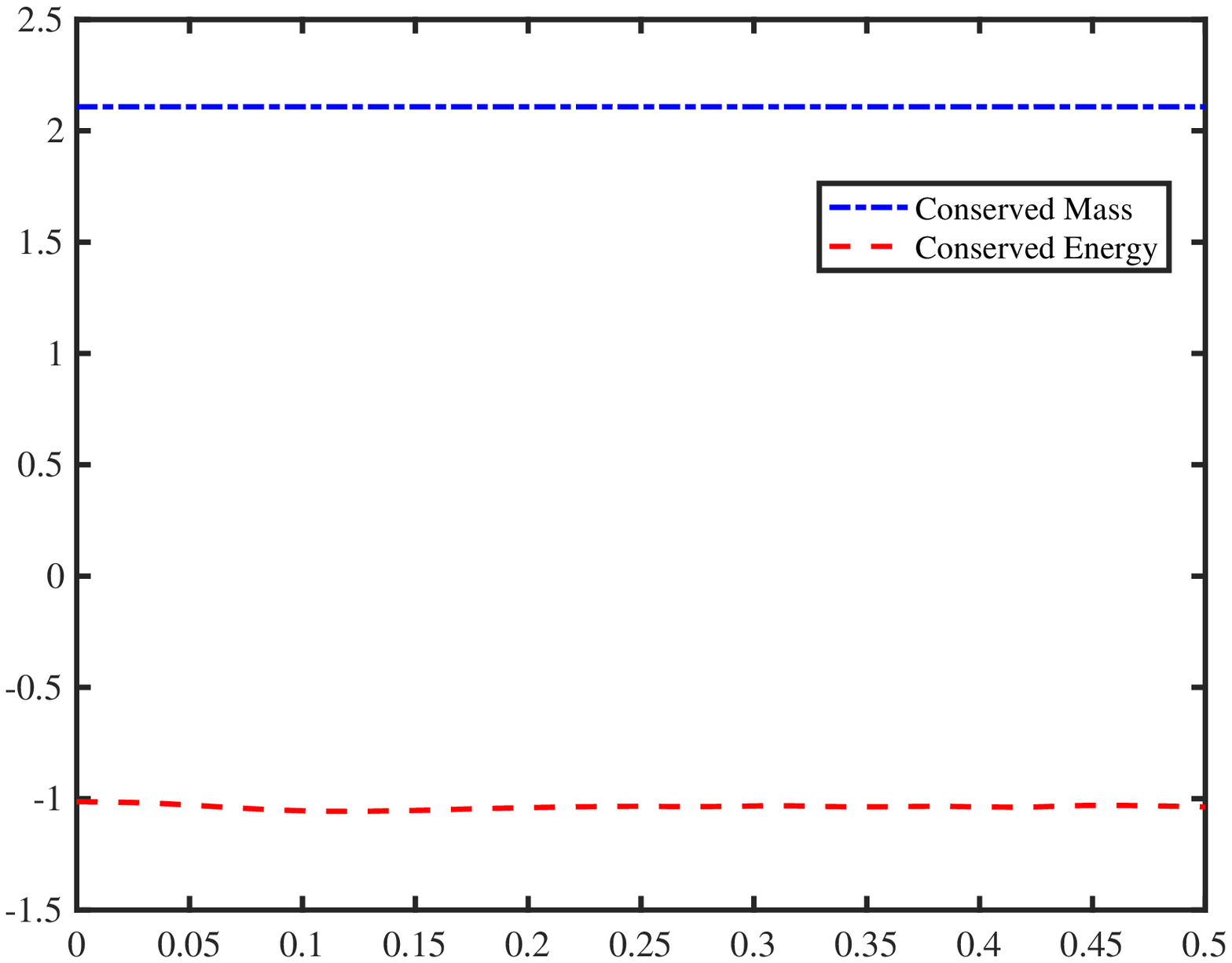} 
\caption{(Top) Time slices of the absolute value of numerical solutions to \eqref{QLS} at $t=0,T/4,T/2,3T/4,T$ with initial data of the form \eqref{pert:ansatz}.  Here, we have taken $2^8$ spatial grid points, $T = .5$, $f = .1 e^{-20 x^2}$.  (Bottom) Tracking the conserved Mass and Energy curves for the simulation.}\label{fig:pertfig}
\end{figure}

\section{Some preliminary estimates}\label{sect:Prelims}

Given \(\tau>0\) we define the Fourier multipliers
\[
\cC_\tau := \cosh(\tau D_y),\qquad \cS_\tau:=i\sinh(\tau D_y),\qquad \cT_\tau:=i\tanh(D_y).
\]
We observe that these multipliers map real-valued functions to real-valued functions and that \(\cC_\tau\) is symmetric whereas \(\cS_\tau,\cT_\tau\) are skew-symmetric. Further, we have the product rules
\begin{equation}\label{ProductRule}
\cC_\tau(fg) = \cC_\tau f\cdot \cC_\tau g - \cS_\tau f\cdot \cS_\tau g,\qquad \cS_\tau(fg) = \cS_\tau f\cdot \cC_\tau g + \cC_\tau f\cdot \cS_\tau g,
\end{equation}
which follow (for sufficiently smooth functions) from taking the Fourier transform and applying hyperbolic trigonometric identities.

Using Plancherel's Theorem, we have the following lemma:
\begin{lem}[An equivalent norm]
We have the estimate
\begin{equation}\label{EquivalentNorm}
\|f\|_{AH^s_\tau}\approx \|\cC_\tau f\|_{H^s},
\end{equation}
uniformly in \(s,\tau\).
\end{lem}
\begin{proof}
Using that \(\cC_\tau = \frac12e^{\tau D_y} + \frac12e^{-\tau D_y}\) we have
\[
\|\cC_\tau f\|_{H^s}\leq \tfrac12 \|f\|_{AH^s_\tau}.
\]
Conversely, we observe that
\(
e^{\pm\tau D_y} = (1 \mp  i\cT_\tau)\cC_\tau
\)
and by Plancherel's Theorem we have
\[
\|(1 \mp i\cT_\tau)\|_{H^s\rightarrow H^s}\leq \|1 \pm \tanh(\cdot)\|_{L^\infty}\leq 2.
\]
Consequently, we may bound
\[
\|f\|_{AH^s}\leq 4\|\cC_\tau f\|_{H^s}.
\]
\end{proof}

We take \(\varphi\in \Test\) to be an even function, identically \(1\) on \([-1,1]\) and supported in \((-2,2)\). We define the Littlewood-Paley projection \(P_0 = \varphi(D_y)\) and for \(j\geq 1\) we define \(P_j = \varphi(2^{-j}D_y) - \varphi(2^{1-j}D_y)\). We then have the following Sobolev-type estimate:
\begin{lem}
If \(s>0\) we have the estimate
\begin{equation}\label{Z-Sobolev}
\|f\|_{L^\infty}\lesssim \|P_0f\|_{L^\infty} + \|f_y\|_{H^{s-\frac12}},
\end{equation}
and identical bounds hold with \(L^\infty\), \(H^{s-\frac12}\) replaced by \(AL^\infty_\tau\), \(H^{s-\frac12}_\tau\), uniformly in \(\tau\).
\end{lem}
\begin{proof}
We decompose by frequency and then apply Bernstein's inequality to bound
\[
\|f\|_{L^\infty}\leq \sum\limits_{j=0}^\infty\|P_j f\|_{L^\infty} \lesssim \|P_0f\|_{L^\infty} + \sum\limits_{j=1}^\infty2^{-sj}\|P_j f_y\|_{H^{s-\frac12}}\lesssim\|P_0f\|_{L^\infty} + \|f_y\|_{H^{s-\frac12}}.
\]
Replacing \(f\) by \(e^{\pm \tau D_y}f\), we obtain the corresponding bound with \(L^\infty\), \(H^{s-\frac12}\) replaced by \(AL^\infty_\tau\), \(AH^{s-\frac12}_\tau\) respectively.
\end{proof}

\begin{lem}
For any \(\tau>0\) and \(1\leq p\leq\infty\) we have the estimate
\begin{equation}\label{C-inverse}
\|\cC_\tau^{-1}\|_{L^p\to L^p} \leq 1.
\end{equation}
\end{lem}
\begin{proof}
We compute that the kernel of \(\cC_\tau^{-1}\) is given by \(K(y) = \frac1{2\tau}\sech(\frac{\pi y}{2\tau})\) and hence \(\|K\|_{L^1_y} = 1\). The estimate \eqref{C-inverse} then follows from Young's inequality.
\end{proof}

We will also require the following technical estimate:
\begin{lem}\label{lem:LinearGrowth}
We have the estimate
\begin{equation}\label{LinearGrowth}
\|\cC_\tau f - f\|_{L^\infty}\lesssim \tau \|f_y\|_{AL^\infty_\tau},
\end{equation}
uniformly for \(0<\tau\leq 1\).
\end{lem}
\begin{proof}
Let \(J = \lfloor -\ln \tau\rfloor\). For high frequencies we bound
\[
\|P_{>J}\cC_\tau f\|_{L^\infty}\lesssim \sum\limits_{j>J} 2^{-j}\|P_j\cC_\tau f_y\|_{L^\infty}\lesssim \tau \|f_y \|_{AL^\infty_\tau}.
\]
Further, from the estimate \eqref{C-inverse} we have
\[
\|P_{>J}f\|_{L^\infty}\lesssim \|P_{>J}\cC_\tau f\|_{L^\infty}\lesssim \tau \|f_y\|_{AL^\infty_\tau}.
\]

For low frequencies, we observe that the kernel of the operator \((\cC_\tau^{-1} - 1)P_{\leq J}\) is \(K'(y)\) where
\[
K(y) = \frac1{2\pi}\int \frac{\sech(\tau\xi) - 1}{i\xi}\varphi(2^{-J}\xi)e^{iy\xi}\,d\xi
\]
is a Schwartz function satisfying \(\|K\|_{L^1}\lesssim\tau\). Consequently, we may apply Young's inequality to obtain the estimate
\[
\|P_{\leq J}\cC_\tau f - P_{\leq J}f\|_{L^\infty}\lesssim \tau\|\cC_\tau f_y\|_{L^\infty}\lesssim \tau\|f_y\|_{AL^\infty_\tau}.
\]
Combining these bounds we obtain the estimate \eqref{LinearGrowth}.
\end{proof}

Our linear estimates will take advantage of the smoothing effect gained from allowing the radius of analyticity to shrink. Observing that
\begin{equation}\label{TimeDerivative}
\frac d{dt}\cC_\tau = -\dot \tau \partial_y\cS_\tau,
\end{equation}
we are motivated to prove the following:
\begin{lem}
We have the estimate
\begin{equation}\label{SmoothingEffect}
\|\cC_\tau f\|_{H^{\frac12}}^2 \lesssim - \Re\left\<\partial_y \cS_\tau f,\cC_\tau f\right\> + \tfrac1\tau \|\cC_\tau f\|_{L^2}^2,
\end{equation}
uniformly for \(0<\tau\leq 1\).
\end{lem}
\begin{proof}
Using Plancherel's Theorem we have
\[
- \Re\left\<\partial_y \cS_\tau f,\cC_\tau f\right\> = \int \xi \tanh(\tau \xi)|\cosh(\tau \xi)\hat f(\xi)|^2\,d\xi.
\]
The estimate \eqref{SmoothingEffect} then follows from the fact that
\[
\<\xi\> \lesssim \xi\tanh(\tau \xi) + \tfrac1\tau,
\]
where the constant can be chosen independently of \(\tau\).
\end{proof}

We define the low-high and high-high paraproduct operators to be
\[
\bT_fg := \sum\limits_{j\geq 4}P_{\leq j-4}f\cdot P_jg,\qquad \bPi[f,g] := fg - \bT_fg - \bT_gf.
\]
We then have the following nonlinear estimates, the proof of which is delayed to Appendix~\ref{app:Multilinear}:
\begin{prop}\label{prop:NonlinearEstimates}
We have the following estimates:
\begin{enumerate}
\item {Symmetric product bounds.} If \(s\geq 0\) then
\begin{align}
\|fg_y\|_{H^s}&\lesssim \|f\|_{L^\infty}\|g_y\|_{H^s} + \|f_y\|_{H^s}\|g\|_{L^\infty},\label{Symm-1}\\
\|fg\|_{Z^s}&\lesssim \|f\|_{L^\infty}\|g\|_{Z^s} + \|f\|_{Z^s}\|g\|_{L^\infty},\label{Symm-2}
\end{align}
\item {Asymmetric product bounds.} If \(s\in \R\) then
\begin{align}
\|\bT_f g_y\|_{H^{s-\frac12}} &\lesssim \|f\|_{L^\infty}\|g_y\|_{H^{s-\frac12}},\label{Asym-1}
\end{align}
if \(0\leq s \leq 1\) then
\begin{align}
\|f g_y - \bT_fg_y\|_{H^s} &\lesssim \|f\|_{W^{1,\infty}}\|g\|_{H^s},\label{Asym-2}
\end{align}
and if \(0 < s\leq \frac12\) then
\begin{align}
\|fg_y\|_{H^{s-\frac12}} &\lesssim \|f\|_{Z^0}\|g_y\|_{H^{s-\frac12}}.\label{Asym-3}
\end{align}
\item {Trilinear bounds.} If \(s\geq 0\) then
\begin{align}
\|fgh\|_{H^s} &\lesssim \|f\|_{H^{s + \frac12}}\|g\|_{H^{\frac12}}\|h\|_{L^2} + \|f\|_{H^{\frac12}}\|g\|_{L^2}\|h\|_{H^{s + \frac12}} + \|f\|_{L^2}\|g\|_{H^{s + \frac12}}\|h\|_{H^{\frac12}},\label{Trilinear}\\
\|fgh\|_{H^s} &\lesssim \|f\|_{H^{s + \frac12}}\|g\|_{L^2}\|h\|_{L^\infty} + \|f\|_{L^2}\|g\|_{H^{s + \frac12}}\|h\|_{L^\infty} + \|f\|_{L^2}\|g\|_{L^2}\|h_y\|_{H^s}.\label{Trilinear-2}
\end{align}
\item {Commutator bounds.} If \(-\frac12<s\leq 1\) and \(j\geq 4\) then
\begin{align}
\|[\<D_y\>^s,f]g_y\|_{L^2}&\lesssim \|f_y\|_{Z^0}\|g\|_{H^s},\label{Comm}\\
\|[P_{\leq j},f]g_y\|_{L^\infty} &\lesssim \|f_y\|_{L^\infty}\|g\|_{L^\infty}\label{Comm-2}
\end{align}

\end{enumerate}
In all cases, identical bounds hold with \(H^s\) replaced by \(AH^s_\tau\), etc., uniformly in \(\tau\).
\end{prop}

\section{Linear estimates}\label{sect:Linear}
In this section we prove a priori estimates for model equations that will subsequently be applied to obtain bounds for \(U\) and \(W\).

We first consider estimates for solutions \(\bz\colon [0,T]\times \R\rightarrow \C\) of the (regularized) linear Schr\"odinger equation
\begin{equation}\label{ToyModel}
i\left(\bz_t + P_{\leq j}(\bB \bz_{\leq j,y})\right)  = P_{\leq j}\bz_{\leq j,yy} + \bff + \bg,
\end{equation}
where we write \(\bz_{\leq j} = P_{\leq j}\bz\) and assume that the coefficient \(\bB\) is real-valued (and not necessarily defined by \eqref{b-def}). We then have the following proposition:
\begin{prop}\label{prop:ModelAP}
Let \(-\frac12< s\leq 1\) and \(0<\delta\ll1\) be a sufficiently small constant. Given \(0<\tau_0\leq 1\), \(M >0\) and \(0<T\leq \frac\delta{2M}\) define
\begin{equation}\label{tau-def}
\tau(t) := \tau_0\left(1 - \tfrac M\delta t\right)\qtq{for}0\leq t\leq T.
\end{equation}
Suppose that for almost every \(t\in(0,T)\) the function \(b(t)\in L^\infty\) satisfies
\begin{align}\label{coefficient-ests}
\bB(t,0) = 0\qtq{and}\|\bB_y(t)\|_{AZ_\tau^0}\leq M.
\end{align}
Let \(\bz\in \Cont([0,T];AH^s_\tau)\) and suppose that for almost every \(t\in(0,T)\) we have \(\bz_t,\bz_{yy}\in AH^s_\tau\) and \(\bz\) satisfies \eqref{ToyModel} with initial data \(\bz(0) = \bz_0\).
Then we have the a priori estimate
\begin{equation}\label{ModelAP}
\sup\limits_{t\in[0,T]}\|\bz\|_{AH^s_\tau} + \sqrt{M\tau_0}\|\bz\|_{L^2_TAH^{s+\frac12}_\tau}\lesssim \|\bz_0\|_{AH^s_{\tau_0}} + \|\bff\|_{L^1_TAH^s_\tau} + \tfrac 1 {\sqrt{M\tau_0}}\|\bg\|_{L^2_TAH^{s-\frac12}_\tau},
\end{equation}
where the implicit depends only on \(\delta\), \(s\).
\end{prop}
\begin{rem}
The a priori estimate of Proposition~\ref{prop:ModelAP} (and Proposition~\ref{prop:TransportModelAP} below) will be applied as part of a bootstrap estimate to prove Theorem~\ref{thrm:UW-existence}. To clarify the role of each constant in its application to the proof of Theorem~\ref{thrm:UW-existence} and explain why the quantifiers appear in the order above, we briefly explain how each constant arises. We first note that \(s,\tau_0\) will be provided by the hypothesis of Theorem~\ref{thrm:UW-existence}, and the constant \(M\) will be determined by \(s\), \(\tau_0\), and the initial data. The constant \(\delta\), which depends only on \(s\), arises from the proof of Proposition~\ref{prop:ModelAP} and ensures that the linear decay of the function \(\tau\) gives us control the second term on \(\LHS{ModelAP}\). The timescale \(T\) will be given to us as part of the bootstrap, but the upper bound of \(\frac\delta{2M}\) is chosen to ensure that for all \(t\in[0,T]\) we have \(\tau(t)\approx \tau_0\), uniformly in all other parameters.
\end{rem}

\begin{rem}
The second term on \(\LHS{ModelAP}\) provides a global smoothing estimate for the equation \eqref{ToyModel}. We remark that this estimate is distinct from the local smoothing effect of linear Schr\"odinger operators, although we expect that solutions of \eqref{ToyModel} do indeed exhibit some form of local smoothing (see for example~\cite{MR2955206,MR2096797}).
\end{rem}

\begin{proof}
We first consider the case that \(s = 0\) and \(j = \infty\), with the convention that \(P_{\leq \infty} = 1\). Using the product rule \eqref{ProductRule} and integration by parts we compute
\begin{equation}\label{EnergyID}
\begin{aligned}
\tfrac d{dt} \tfrac12\|\cC_\tau \bz\|_{L^2}^2 - \tfrac {M\tau_0}\delta\Re\bigl\<\partial_y\cS_\tau \bz,\cC_\tau\bz\bigr\> &= \tfrac12\Re\bigl\<\cC_\tau\bB_y \cdot\cC_\tau\bz,\cC_\tau\bz\bigr\> + \Re\bigl\<\cS_\tau\bB\cdot\cS_\tau\bz_y,\cC_\tau\bz\bigr\>\\
&\quad +\Im\bigl\<\cC_\tau\bff,\cC_\tau\bz\bigr\>+\Im\bigl\<\cC_\tau\bg,\cC_\tau\bz\bigr\>.
\end{aligned}
\end{equation}
We remark that in order to justify the integration by parts we use \eqref{LinearGrowth} to bound
\[
\|\cC_\tau \bB\|_{L^\infty}\lesssim \|\bB\|_{L^\infty} + \tau \|\bB_y\|_{AL^\infty_\tau} <\infty
\]
for a.e. \(t\in(0,T)\). This is the only place in the proof we use that \(\bB\) is bounded and hence the estimate \eqref{ModelAP} is independent of the size of \(\|\bB\|_{L^\infty}\).

For the smoothing term on LHS\eqref{EnergyID} we apply the estimate \eqref{SmoothingEffect} to bound
\[
\tfrac{M\tau_0}\delta\|\cC_\tau\bz\|_{H^{\frac12}}^2 \lesssim - \tfrac {M\tau_0}\delta\Re\bigl\<\partial_y\cS_\tau \bz,\cC_\tau\bz\bigr\> + \tfrac M\delta\|\cC_\tau\bz\|_{L^2}^2,
\]
where we have used that \(\tau\approx \tau_0\).

For the first term on RHS\eqref{EnergyID} we use the hypothesis \eqref{coefficient-ests} to bound
\[
\left|\bigl\<\cC_\tau\bB_y \cdot\cC_\tau\bz,\cC_\tau\bz\bigr\>\right|\leq \|\cC_\tau \bB_y\|_{L^\infty}\|\cC_\tau\bz\|_{L^2}^2\lesssim M\|\cC_\tau \bz\|_{L^2}^2.
\]

For the second term on RHS\eqref{EnergyID} we decompose using paraproducts to write
\begin{equation}
\label{RHS21}
\Re\bigl\<\cS_\tau\bB\cdot\cS_\tau\bz_y,\cC_\tau\bz\bigr\> = \Re\bigl\<\bT_{\cS_\tau\bB}\cS_\tau\bz_y,\cC_\tau\bz\bigr\> + \Re\bigl\<\cS_\tau\bB\cdot\cS_\tau\bz_y - \bT_{\cS_\tau\bB}\cS_\tau\bz_y,\cC_\tau\bz\bigr\>.
\end{equation}
Further, we observe that applying Bernstein's inequality with the fact that \(\cS_\tau = \frac12(e^{\tau D_y} - e^{-\tau D_y})\) and \(\tau\approx \tau_0\) we may bound
\[
\|\cS_\tau \bB\|_{L^\infty} \lesssim \tau \Bigl(\|P_{\leq \frac1\tau}\bB_y\|_{A L^\infty_\tau} + \|P_{>\frac1\tau}\bB_{yy}\|_{AH^{-\frac12}_\tau}\Bigr)\lesssim \tau_0M,\qquad\|\cS_\tau \bB_y\|_{L^\infty}\lesssim \|\bB_y\|_{A L^\infty_\tau}\lesssim M.\]
Consequently, the first term in \eqref{RHS21} may be bounded by applying the estimate \eqref{Asym-1} with the fact that \(\|\cT_\tau\|_{L^2\rightarrow L^2}\leq 1\) (which follows from Plancherel's Theorem) to obtain
\[
\left|\bigl\<\bT_{\cS_\tau\bB}\cS_\tau\bz_y,\cC_\tau\bz\bigr\>\right|\lesssim \|\cS_\tau \bB\|_{L^\infty}\|\cS_\tau \bz_y\|_{H^{-\frac12}}\|\cC_\tau\bz\|_{H^{\frac12}}\lesssim \tau_0M \|\cC_\tau \bz\|_{H^{\frac12}}^2,
\]
whereas the second term in \eqref{RHS21} may be bounded by applying \eqref{Asym-2} to obtain
\[
\left|\bigl\<\cS_\tau\bB\cdot\cS_\tau\bz_y - \bT_{\cS_\tau\bB}\cS_\tau\bz_y,\cC_\tau\bz\bigr\>\right|\lesssim\|\cS_\tau \bB\|_{W^{1,\infty}}\|\cS_\tau \bz\|_{L^2}\|\cC_\tau \bz\|_{L^2}\lesssim M\|\cC_\tau \bz\|_{L^2}^2.
\]

For the remaining terms on RHS\eqref{EnergyID} we use duality to bound
\[
\left|\bigl\<\cC_\tau\bff,\cC_\tau\bz\bigr\>\right| \leq \|\cC_\tau \bff\|_{L^2}\|\cC_\tau \bz\|_{L^2},\qquad \left|\bigl\<\cC_\tau\bg,\cC_\tau\bz\bigr\>\right| \leq \|\cC_\tau \bg\|_{H^{-\frac12}}\|\cC_\tau \bz\|_{H^{\frac12}}.
\]

Combining these estimates we obtain
\[
\partial_t \|\cC_\tau \bz\|_{L^2}^2 + \tfrac{M\tau_0}\delta\|\cC_\tau \bz\|_{H^{\frac12}}^2 \lesssim \tfrac M\delta \|\cC_\tau \bz\|_{L^2}^2 + \tau_0M\|\cC_\tau \bz\|_{H^{\frac12}}^2 + \|\cC_\tau \bff\|_{L^2}\|\cC_\tau \bz\|_{L^2} + \|\cC_\tau \bg\|_{H^{-\frac12}}\|\cC_\tau \bz\|_{H^{\frac12}}.
\]
Taking \(C>0\) to be a sufficiently large (absolute) constant to absorb the first term on the right hand side, we then obtain
\begin{align*}
&\partial_t\left(e^{-C t\frac M\delta}\|\cC_\tau \bz\|_{L^2}^2 + \tfrac{M\tau_0}\delta\int_0^t e^{-C\sigma\frac M\delta}\|\cC_\tau \bz\|_{H^{\frac12}}^2\,d\sigma,\right)\\
&\qquad \lesssim e^{-Ct\frac M\delta}\Biggl(\tau_0M\|\cC_\tau \bz\|_{H^{\frac12}}^2 + \|\cC_\tau \bff\|_{L^2}\|\cC_\tau \bz\|_{L^2} + \|\cC_\tau \bg\|_{H^{-\frac12}}\|\cC_\tau \bz\|_{H^{\frac12}}\Biggr).
\end{align*}
We then integrate, using that \(T\frac M\delta\leq\frac12\), to obtain
\begin{align*}
\sup\limits_{t\in[0,T]}\|\cC_\tau \bz\|_{L^2}^2 + \tfrac{M\tau_0}\delta\|\cC_\tau \bz\|_{L^2_TH^{\frac12}}^2&\lesssim \|\cC_{\tau_0}\bz_0\|_{L^2}^2 + \tau_0M\|\cC_\tau\bz\|_{L^2_TH^{\frac12}}^2\\
&\quad + \|\cC_\tau\bff\|_{L^1_TL^2}\|\cC_\tau\bz\|_{L^\infty_TL^2} + \|\cC_\tau\bg\|_{L^2_TH^{-\frac12}}\|\cC_\tau \bz\|_{L^2_TH^{\frac12}}\\
&\lesssim\|\cC_{\tau_0}\bz_0\|_{L^2}^2 +  \delta \left(\sup\limits_{t\in[0,T]}\|\cC_\tau \bz\|_{L^2}^2 + \tfrac{M\tau_0}\delta \|\cC_\tau\bz\|_{L^2_TH^{\frac12}}^2\right)\\
&\quad + \tfrac1\delta\left(\|\cC_\tau \bff\|_{L^1_TL^2}^2 + \tfrac \delta {M\tau_0}\|\cC_\tau \bg\|_{L^2_TH^{-\frac12}}^2\right),
\end{align*}
so provided \(0<\delta\ll1\) is sufficiently small, independently of all other parameters, we may apply the estimate \eqref{EquivalentNorm} to obtain the estimate \eqref{ModelAP}.

To handle the case \(s\neq 0\) we simply apply \(\<D_y\>^s\) to the equation \eqref{ToyModel} and apply the commutator estimate \eqref{Comm} to obtain
\[
\|[\<D_y\>^s,\bB]\bz_y\|_{AL^2_\tau}\lesssim \|\bB_y\|_{AZ^0_\tau}\|\bz\|_{AH^s_\tau},
\]
where we note that we have used that \(-\frac12< s\leq 1\). The estimate \eqref{ModelAP} then follows from possibly shrinking the size of \(\delta\), depending on \(s\). Finally, to handle the case that \(j<\infty\), we simply replace \(\bz\) by \(\bz_{\leq j}\) on the right hand side of \eqref{EnergyID} and use that \(\|\bz_{\leq j}\|_{AH^s_{\tau}}\leq\|\bz\|_{AH^s_\tau}\).
\end{proof}

The second model we consider is the (regularized) transport-type equation
\begin{equation}\label{TransportModel}
\bz_t + P_{\leq j}(\bB\bz_{\leq j,y}) = \bff.
\end{equation}
This will be applied to bound the low frequencies of \(W\).

In order to prove a priori bounds for \eqref{TransportModel}, it will be useful to first record the following consequence of the Picard-Lindel\"of Theorem:
\begin{lem}\label{lem:Picard}
Let \(T>0\) and \(\bB\colon[0,T]\times\R\to\R\) be a continuous function so that for some constant \(C\geq 0\) we have \(|\bB(t,0)| \leq  C\) and \(\bB_y\in L^\infty([0,T]\times\R)\). Then, for each \(y\in \R\), there exists a unique solution \(Y\in \Cont^1([0,T])\) of the ODE
\begin{equation}\label{Transport}
\begin{cases}
Y_t(t,y) = \bB(t,Y(t,y))
\smallskip\\
Y(0,y) = y.
\end{cases}
\end{equation}
Further, for each \(y\in \R\) the derivative \(Y_y\in \Cont^1([0,T])\) and we have the estimate
\begin{equation}\label{ODE-growth}
e^{-T\|\bB_y\|_{L^\infty_TL^\infty}}\leq \|Y_y\|_{L^\infty_TL^\infty}\leq e^{T\|\bB_y\|_{L^\infty_TL^\infty}},
\end{equation}
and hence the map \(y\to Y\) is a diffeomorphism.
\end{lem}
\begin{proof}
Our hypotheses ensure that \(\bB(t,y)\) is continuous in \(t\) and Lipschitz continuous in \(y\), with uniform Lipschitz constant \(\|\bB_y\|_{L^\infty_TL^\infty}\). For fixed \(y\in \R\), the Picard-Lindel\"of Theorem then guarantees local existence for \eqref{Transport}. However, using the estimate
\[
|b(t,Y(t,y)|\leq C + \|\bB_y\|_{L^\infty_TL^\infty}|Y(t,y)|,
\]
and Gronwall's inequality, the solution can be extended to the entire time interval \([0,T]\). It remains to prove \eqref{ODE-growth}. However, this readily follows from the observation that
\[
\frac d{dt}\bigl(\log Y_y(t,y)\bigr) = b_y(t,Y(t,y)).
\]
\end{proof}

We may then prove our main a priori estimate for solutions of \eqref{TransportModel}:
\begin{prop}\label{prop:TransportModelAP}
Let \(0<\delta\ll1\) be a sufficiently small constant and \(j\geq 8\). Given \(0<\tau_0\leq 1\), \(M>0\) and \(0<T\leq \frac\delta{2M}\), define \(\tau\in\Cont^1([0,T])\) as in \eqref{tau-def} and suppose that \(\bB\colon[0,T]\times\R\to\R\) is a continuous function satisfying \eqref{coefficient-ests}. Then, if \(\bz\in \Cont([0,T];AL^\infty_\tau)\) such that for almost every \(t\in(0,T)\) we have that \(\bz_t,\bz_y\in AL^\infty_\tau\) and \(\bz\) is a solution of \eqref{TransportModel} with initial data \(\bz(0) = \bz_0\) we have estimate
\begin{equation}\label{TransportModelAP}
\sup\limits_{t\in[0,T]}\|P_0\bz\|_{AL^\infty_\tau}\lesssim \|P_0\bz_0\|_{AL^\infty_{\tau_0}} + \|P_{>0}\bz\|_{L^\infty_TAL^\infty_\tau} + \|P_0\bff\|_{L^1_TAL^\infty_\tau}
\end{equation}
where the implicit constant depends only on \(\delta\).
\end{prop}
\begin{proof}
Applying \(P_0\) to the equation \eqref{TransportModel} and using that for \(j\geq 8\) we have \(P_0P_{\leq j} = P_0P_{\leq 4} = P_0\), and \(P_{\leq 4}P_{>j} = 0\) we obtain
\[
P_0\bz_t + \bB P_0\bz_y = - [P_0,\bB]\bz_y + P_0[P_{\leq 4},\bB]\bz_{>j,y} + P_0\bff,
\]
where \(\bz_{>j} = P_{>j}\bz\).

From \eqref{C-inverse} and \eqref{coefficient-ests} we have
\[
\|\bB_y\|_{L^\infty_TL^\infty}\lesssim \|\cC_\tau \bB_y\|_{L^\infty_TL^\infty}\lesssim M.
\]
Applying the estimates \eqref{Comm-2} and \eqref{C-inverse}, we may then bound
\begin{align*}
\|[P_0,\bB]\bz_y\|_{L^1_TL^\infty} + \|P_0[P_{\leq 4},\bB]\bz_{>j,y}\|_{L^1_TL^\infty} + \|P_0\bff\|_{L^1_TL^\infty} &\lesssim T\|\bB_y\|_{L^\infty_TL^\infty}\|\bz\|_{L^\infty_TL^\infty} + \|P_0\bff\|_{L^1_TL^\infty}\\
&\lesssim \delta \|\bz\|_{L^\infty_TAL^\infty_\tau} + \|P_0\bff\|_{L^1_TAL^\infty_\tau}.
\end{align*}

Applying Lemma~\ref{lem:Picard} we may find a solution \(Y\) of \eqref{Transport} so that the map \(y\to Y\) is a diffeomorphsim. Writing \((\bz\circ Y)(t,y) = \bz(t,Y(t,y))\), etc. we obtain
\[
\partial_t((P_0\bz)\circ Y) = (- [P_0,\bB]\bz_y + P_0[P_{\leq 4},\bB]z_{>j,y} + P_0\bff)\circ Y.
\]
As a consequence, we may bound
\begin{align*}
\sup\limits_{t\in[0,T]}\|P_0\bz\|_{L^\infty}&\lesssim \|\bz_0\|_{L^\infty} + \|[P_0,\bB]\bz_y\|_{L^1_TL^\infty} + \|P_0[P_{\leq 4},\bB]\bz_{>j,y}\|_{L^1_TL^\infty} + \|P_0\bff\|_{L^1_TL^\infty}\\
&\lesssim \|\bz_0\|_{L^\infty_TAL^\infty_\tau} + \delta\|\bz\|_{L^\infty_TAL^\infty_\tau} + \|P_0\bff\|_{L^1_TAL^\infty_\tau}.
\end{align*}
Using that \(e^{\pm\tau D_y}P_{\leq 4}\) is bounded on \(L^\infty\) and \(P_{\leq 4}P_0 = P_0\) we may then bound
\[
\sup\limits_{t\in[0,T]}\|P_0\bz\|_{AL^\infty_\tau} \lesssim \sup\limits_{t\in[0,T]}\|P_0\bz\|_{L^\infty}.
\]
Finally, we split
\[
\|\bz\|_{AL^\infty_\tau}\leq \|P_0\bz\|_{AL^\infty_\tau} + \|P_{>0}\bz\|_{AL^\infty_\tau}
\]
so that, by choosing \(0<\delta\ll1\) sufficiently small, we obtain the estimate \eqref{TransportModelAP}.
\end{proof}

\section{Proof of Theorem~\ref{thrm:UW-existence}}\label{sect:Existence}
In this section we prove the existence of a solution \((U,W)\) of the system \eqref{U-eqn}, \eqref{W-eqn} by taking the (weak) limit of a sequence of solutions to a sequence of regularized systems. In our proof of existence we will consider \(U,W\) to be independently defined functions, i.e. not necessarily satisfying the identity \eqref{UtoW}. Once we have proved the existence of a solution to the system \eqref{U-eqn}, \eqref{W-eqn}, it is clear that if the initial data satisfies \eqref{UtoW} then the corresponding solution must also satisfy this identity.

Our regularization of the system \eqref{U-eqn}, \eqref{W-eqn} is the following:
\begin{subequations}\label{RegularizedSystem}
\begin{align}
\!iU_t &= P_{\leq j}\Bigl[- iBU_{\leq j,y} + U_{\leq j,yy} + 2i\beta_{\leq j} U_{\leq j,y} + \mu |U_{\leq j}|^2U_{\leq j}\Bigr],\label{RegularizedU}
\\
\!iW_t &=  P_{\leq j}\Bigl[- i BW_{\leq j,y} + W_{\leq j,yy} + \bigl(2(W_{\leq j})^2 - \tfrac12|W_{\leq j}|^2\bigr)_y + 3i\alpha_{\leq j}\beta_{\leq j}W_{\leq j}  + 2\mu |U_{\leq j}|^2\alpha_{\leq j}\Bigr],\label{RegularizedW}
\end{align}
\end{subequations}
where we denote \(f_{\leq j} = P_{\leq j}f\), take \(\alpha,\beta\) to be the real and imaginary parts of \(W\) as in \eqref{alpha-beta-def}, and define the regularized velocity
\[
B(t,y;j) = - 3\sech(2^{-j}y)\int_0^y \alpha_{\leq j}(t,\zeta)\beta_{\leq j}(t,\zeta)\,d\zeta.
\]
We note that the velocity \(b\) is expected to have linear (in \(y\)) growth as \(|y|\to\infty\). In the regularized version \(B\) we introduce an additional spatial weight to ensure that the velocity is bounded, albeit with a bound that depends on \(j\).

We first prove the existence of a solution to the regularized system \eqref{RegularizedSystem}:

\begin{lem}\label{lem:ODEExistence}
Given \(0<s\leq \frac12\), \(0<\tau_0\leq 1\) and \((U_0,W_0)\in AH^s_{\tau_0}\times AZ^s_{\tau_0}\) there exists a time \(T_0>0\) and a solution \((U,W)\in \Cont^1([0,T_0];AH^s_{\tau_0}\times AZ^s_{\tau_0})\) of \eqref{RegularizedSystem} with initial data \((U,W)(0) = (U_0,W_0)\).
\end{lem}
\begin{proof}
We first bound the velocity \(B\) by
\[
\|B\|_{L^\infty}\lesssim 2^j\|\alpha_{\leq j}\|_{L^\infty}\|\beta_{\leq j}\|_{L^\infty}\lesssim 2^j\|W\|_{L^\infty}^2.
\]

Next, we apply Bernstein's inequality followed by the estimate \eqref{C-inverse} to bound
\begin{align*}
\|\RHS{RegularizedU}\|_{AH^s_{\tau_0}}&\lesssim e^{\tau_02^j}2^{sj}\Bigl[ 2^{2j}\|W\|_{L^\infty}^2\|U\|_{L^2} + 2^{2j}\|U\|_{L^2} + 2^j\|W\|_{L^\infty}\|U\|_{L^2} + 2^j\|U\|_{L^2}^3\Bigr]\\
&\lesssim_{\tau_0,j}\bigl(1 + \|(U,W)\|_{AH^s_{\tau_0}\times AZ^s_{\tau_0}}\bigr)^2\|(U,W)\|_{AH^s_{\tau_0}\times AZ^s_{\tau_0}},\\
\|\RHS{RegularizedW}\|_{AL^\infty_{\tau_0}} &\lesssim e^{\tau_02^j}\Bigl[ 2^{2j}\|W\|_{L^\infty}^3 + 2^{2j}\|W\|_{L^\infty} + 2^j\|W\|_{L^\infty}^2 + \|W\|_{L^\infty}^3 + 2^j\|U\|_{L^2}^2\|W\|_{L^\infty}\Bigr]\\
&\lesssim_{\tau_0,j}\bigl(1 + \|(U,W)\|_{AH^s_{\tau_0}\times AZ^s_{\tau_0}}\bigr)^2\|(U,W)\|_{AH^s_{\tau_0}\times AZ^s_{\tau_0}},\\
\|\partial_y\RHS{RegularizedW}\|_{H^{s-\frac12}} &\lesssim e^{\tau_02^j}\Bigl[2^{2j}\|W\|_{L^\infty}\|W_y\|_{L^2} + 2^{2j}\|W_y\|_{L^2} + 2^j\|W\|_{L^\infty}\|W_y\|_{L^2}\\
&\qquad  + \|W\|_{L^\infty}^2\|W_y\|_{L^2} + 2^{\frac32j}\|W\|_{L^\infty}\|U\|_{L^2}^2\Bigr]\\
&\lesssim_{\tau_0,j}\bigl(1 + \|(U,W)\|_{AH^s_{\tau_0}\times AZ^s_{\tau_0}}\bigr)^2\|(U,W)\|_{AH^s_{\tau_0}\times AZ^s_{\tau_0}}.
\end{align*}

Applying these bounds, and identical bounds for the difference of two solutions, we see that \(\RHS{RegularizedSystem}\) is Lipschitz continuous as a map from \(AH^s_{\tau_0}\times AZ^s_{\tau_0}\) to itself. The proof is then completed by applying the Picard-Lindel\"of Theorem.
\end{proof}

We now turn to the proof of Theorem~\ref{thrm:UW-existence}. Our goal here is to prove uniform (in \(j\)) estimates for solutions of \eqref{RegularizedSystem}. These uniform bounds show that: 1) The solution of \eqref{RegularizedSystem} can be extended to a \(j\)-independent time \(T>0\); 2) We may pass to a (weak) limit as \(j\to \infty\) to obtain a solution of the system \eqref{U-eqn}, \eqref{W-eqn}. As is standard in such arguments, our proof of these uniform bounds takes the form of a bootstrap estimate, relying on the local existence provided by Lemma~\ref{lem:ODEExistence}.
\begin{proof}[Proof of Theorem~\ref{thrm:UW-existence}]
We assume that \(j\geq 8\) and choose a sufficiently small constant \(0<\delta = \delta(s)\ll1\), independent of \(j\), as in the hypotheses of Propositions~\ref{prop:ModelAP},~\ref{prop:TransportModelAP}. Next we choose \(K = K(\delta,s,\tau_0)\geq 1\) and \(M = M(K,\delta,s,\tau_0,U_0,W_0)\geq 1\) to be sufficiently large constants. Given these constants, we set \(T_* = \tfrac \delta{2M}\) and \(\tau\in \Cont^1([0,T])\) as in \eqref{tau-def}. We will subsequently ignore the dependence of the bounds on \(\delta,s\), which we can assume have been fixed.

We make the bootstrap assumption that for some \(0<T\leq T_*\) we have
\begin{equation}\label{Bootstrap-1}
\sup\limits_{t\in[0,T]}\|U\|_{AH^s_\tau} + \sqrt{M\tau_0}\|U\|_{L^2_TAH^{\frac12}_\tau} + \sup\limits_{t\in[0,T]}\|W\|_{AZ^s_\tau} + \sqrt{M\tau_0}\|W_y\|_{L^2_TAL^2_\tau} \leq \tfrac{\sqrt{M}}K.
\end{equation}
Our goal will be to prove that if the solution of \eqref{RegularizedSystem} satisfies \eqref{Bootstrap-1}, then in fact it must satisfy \eqref{Bootstrap-1} with \(\RHS{Bootstrap-1}\) replaced by \(\frac{\sqrt M}{10K}\). By applying Lemma~\ref{lem:ODEExistence}, our solution may then be extended until time \(T = T_*\) and satisfies the estimate \eqref{Bootstrap-1}. We remark that this application of Lemma~\ref{lem:ODEExistence} uses that, given \(0\leq t_1<t_2\leq T_*\), we have
\[
\int_{t_1}^{t_2}\|f\|_{AH^{\frac12}_\tau}^2\,dt\lesssim (\tau_0M)^{2s-1}(t_2 - t_1)^{2s}\sup\limits_{t\in[t_1,t_2]}\|f\|_{AH^s_{\tau(t_1)}}^2.
\]
In particular, this motivates the difference between the pointwise-in-time and \(L^2\)-in-time regularities in \eqref{Bootstrap-1}.

We first consider estimates for the coefficient \(B\). Here it will be useful to denote
\[
f\sbrack 1(y) = - 3 \sech(2^{-j}y),\qquad f\sbrack 2(y) = 3\tanh(2^{-j}y)\sech(2^{-j}y)
\]
so that
\begin{align*}
B_y &= \underbrace{f\sbrack 1 \alpha_{\leq j}\beta_{\leq j}}_{I_1} + \underbrace{2^{-j}f\sbrack 2P_{> 0}\int_0^y\alpha_{\leq j}\beta_{\leq j}\,d\zeta}_{I_2} + \underbrace{2^{-j}f\sbrack 2 P_0\int_0^y\alpha_{\leq j}\beta_{\leq j}\,d\zeta}_{I_3}.
\end{align*}
We observe that for \(\ell = 1,2\) the functions \(f\sbrack \ell(y)\) are analytic on the strip \(\{y:|\Im y|\leq 1\}\subseteq \C\) and hence
\[
e^{\pm\tau D_y} f \sbrack \ell (y) = f\sbrack\ell(y \mp i\tau).
\]
In particular, using the embedding \(L^\infty\cap \dot H^1\subseteq Z^0\), we may use the explicit expressions for \(e^{\pm \tau D_y}f\sbrack \ell\) to obtain the \(j,\tau\)-independent bounds
\[
\|f\sbrack\ell\|_{AZ^0_\tau}\lesssim 1,\qquad \|\<y\> e^{\pm \tau D_y}f\sbrack \ell\|_{Z^0}\lesssim 2^j,
\]
Applying the product estimate \eqref{Symm-2} with the fact that
\[
\left\|P_{> 0}\int_0^y\alpha_{\leq j}\beta_{\leq j}\,d\zeta\right\|_{AZ^0_\tau}\lesssim\|\alpha_{\leq j}\beta_{\leq j}\|_{AZ^0_\tau},
\]
we may then bound
\[
\|I_1\|_{AZ^0_\tau} + \|I_2\|_{AZ^0_\tau}\lesssim \|\alpha_{\leq j}\beta_{\leq j}\|_{AZ^0_\tau}\lesssim\|W\|_{AL^\infty_\tau}\|W\|_{AZ^0_\tau},
\]
uniformly in \(j\). For the remaining term we first write
\[
e^{\pm \tau D_y}P_0\int_0^y\alpha_{\leq j}\beta_{\leq j}\,d\zeta = \int K_\pm(y - z) \left(\int_0^z \alpha_{\leq j}\beta_{\leq j}\,d\zeta\right)\,dz,
\]
where \(K_\pm\) is the kernel of \(e^{\pm \tau D_y}P_0\). As the functions \(K_{\pm}\) are Schwartz, we may bound
\begin{align*}
\left|e^{\pm \tau D_y}P_0\int_0^y\alpha_{\leq j}\beta_{\leq j}\,d\zeta\right|&\lesssim\<y\>\|\alpha_{\leq j}\|_{L^\infty}\|\beta_{\leq j}\|_{L^\infty}\lesssim \<y\>\|W\|_{AL^\infty_\tau}^2\\
\left|\partial_ye^{\pm \tau D_y}P_0\int_0^y\alpha_{\leq j}\beta_{\leq j}\,d\zeta\right|&\lesssim\|\alpha_{\leq j}\|_{L^\infty}\|\beta_{\leq j}\|_{L^\infty}\lesssim \|W\|_{AL^\infty_\tau}^2.
\end{align*}
Again using the embedding \(L^\infty\cap \dot H^1\subseteq Z^0\), with the fact that \(\<y\>^{-1}\in L^2\), this yields the estimate
\[
\left\|\tfrac1{\<y\>}e^{\pm \tau D_y}P_0\int_0^y\alpha_{\leq j}\beta_{\leq j}\,d\zeta\right\|_{Z^0}\lesssim \|W\|_{AL^\infty_\tau}^2,
\]
from which we obtain the \(j\)-independent bound
\[
\|I_3\|_{AZ^0_\tau}\lesssim \|W\|_{AL^\infty_\tau}^2.
\]
Combining these bounds, we obtain
\begin{equation}\label{bigBbound}
\|B_y\|_{AZ_\tau^0}\lesssim \|W\|_{AL^\infty_\tau}\|W\|_{AZ_\tau^0}\lesssim \tfrac M{K^2}.
\end{equation}
In particular, provided \(K\gg1\) is sufficiently large (independently of all other parameters), we see that for all \(t\in[0,T]\) we have
\[
\|B_y\|_{AZ^0_\tau}\leq M,
\]
and hence the hypothesis \eqref{coefficient-ests} of Proposition~\ref{prop:ModelAP} is satisfied (with \(b\) replaced by \(B\)).

Next we consider bounds for \(U\), which we obtain by applying Proposition~\ref{prop:ModelAP}. Applying the estimate \eqref{Asym-3} and the bootstrap assumption \eqref{Bootstrap-1} we may bound
\[
\|\beta_{\leq j} U_{\leq j,y}\|_{L^2_TAH^{s-\frac12}_\tau}\lesssim \|W\|_{L^\infty_TAZ^0_\tau}\|U\|_{L^2_TAH^{s+\frac12}_\tau}\lesssim \tfrac{\sqrt{M}}{K}\|U\|_{L^2_TAH^{s+\frac12}_\tau},
\]
uniformly in \(j\). Using the trilinear estimate \eqref{Trilinear} and the bootstrap assumption \eqref{Bootstrap-1} we estimate
\[
\||U_{\leq j}|^2U_{\leq j}\|_{L^1_TAH^s_\tau}\lesssim \|U\|_{L^\infty_TAL^2_\tau}\|U\|_{L^2_TAH^{\frac12}_\tau}\|U\|_{L^2_TAH^{s+\frac12}_\tau}\lesssim \tfrac{\sqrt M}{K^2\sqrt{\tau_0}}\|U\|_{L^2_TAH^{s+\frac12}_\tau},
\]
uniformly in \(j\). Applying Proposition~\ref{prop:ModelAP}, we may then bound
\[
\sup\limits_{t\in[0,T]}\|U\|_{AH^s_\tau} + \sqrt{M\tau_0}\|U\|_{L^2_TAH^{s+\frac12}_{\tau}}\lesssim \|U_0\|_{AH^s_{\tau_0}} + \bigl(\tfrac1{K\tau_0\sqrt{M}} + \tfrac1{K^2\tau_0}\bigr)\sqrt{M\tau_0}\|U\|_{L^2_TAH^{s+\frac12}_\tau},
\]
uniformly in \(j\).
Provided \(K\) is sufficiently large (depending only on \(\delta,s,\tau_0\)) we obtain the estimate
\begin{equation}\label{U-AP-Bound}
\sup\limits_{t\in[0,T]}\|U\|_{AH^s_\tau} + \sqrt{M\tau_0}\|U\|_{L^2_TAH^{s+\frac12}_{\tau}}\lesssim_K \|U_0\|_{AH^s_{\tau_0}},
\end{equation}
uniformly in \(M\geq 1\) and \(j\). In particular, provided \(M\gg1\) is sufficiently large (depending on \(U_0,K,\delta,s,\tau_0\)) we have
\[
\sup\limits_{t\in[0,T]}\|U\|_{AH^s_\tau} + \sqrt{M\tau_0}\|U\|_{L^2_TAH^{\frac12}_{\tau}} \leq\tfrac{\sqrt M}{20K},
\]
which closes the first part of the bootstrap.

Next we consider bounds for \(W_y\), which will once again be proved using Proposition~\ref{prop:ModelAP}. Differentiating \eqref{RegularizedW}, we obtain the equation
\begin{align*}
i\left(W_{ty} + P_{\leq j}(BW_{\leq j,yy})\right) &= P_{\leq j}W_{\leq j,yyy}\\
&\quad + P_{\leq j}\Bigl[4W_{\leq j}W_{\leq j,y} - \tfrac12\bar W_{\leq j}W_{\leq j,y} - \tfrac 12 W_{\leq j}\bar W_{\leq j,y} + 2\mu |U_{\leq j}|^2\alpha_{\leq j}\Bigr]_y\\
&\quad + iP_{\leq j}\Bigl[3\alpha_{\leq j,y}\beta_{\leq j} W_{\leq j} + 3\alpha_{\leq j}\beta_{\leq j,y}W_{\leq j} + 3\alpha_{\leq j}\beta_{\leq j}W_{\leq j,y} - B_yW_{\leq j,y} \Bigr].
\end{align*}
Applying the estimate \eqref{Symm-1} we may bound
\[
\|(4W_{\leq j}W_{\leq j,y} - \tfrac12\bar W_{\leq j}W_{\leq j,y} - \tfrac 12 W_{\leq j}\bar W_{\leq j,y})_y\|_{L^2_TAH^{s-1}_\tau}\lesssim \|W\|_{L^\infty_TAL^\infty_\tau}\|W_y\|_{L^2_TAH^s_\tau}\lesssim \tfrac{\sqrt M}K\|W_y\|_{L^2_TAH^s_\tau},
\]
uniformly in \(j\). Similarly, using that \(T\lesssim \frac1M\), with \eqref{Symm-2}, \eqref{Asym-3} we obtain
\begin{align*}
&\|\alpha_{\leq j}\beta_{\leq j,y}W_{\leq j} + \alpha_{\leq j,y}\beta_{\leq j} W_{\leq j} + \alpha_{\leq j}\beta_{\leq j}W_{\leq j,y}\|_{L^1_TAH^{s-\frac12}_\tau}\\
&\qquad\lesssim T\|W\|_{L^\infty_TAL^\infty_\tau}\|W\|_{L^\infty_TAZ^0_\tau}\|W_y\|_{L^\infty_TAH^{s-\frac12}_\tau}\\
&\qquad\lesssim \tfrac 1{K^2}\|W_y\|_{L^\infty_TAH^{s-\frac12}_\tau},
\end{align*}
and using the estimate \eqref{bigBbound} with \eqref{Asym-3}, we also have
\[
\|B_yW_{\leq j,y}\|_{L^1_TAH^{s-\frac12}_\tau}\lesssim T\|B_y\|_{L^\infty_TAZ^0_\tau}\|W_y\|_{L^\infty_TAH^{s-\frac12}_\tau}\lesssim \tfrac 1{K^2}\|W_y\|_{L^\infty_TAH^{s-\frac12}_\tau},
\]
where both estimates are again uniform in \(j\). Finally, applying the estimate \eqref{Trilinear-2} we may bound
\begin{align*}
\|( |U_{\leq j}|^2\alpha_{\leq j})_y\|_{L^2_TAH^{s - 1}_\tau}&\lesssim \|U\|_{L^2_TAH^{s+\frac12}_\tau}\|U\|_{L^\infty_TAL^2_\tau}\|W\|_{L^\infty_TAL^\infty_\tau} + \|U\|_{L^\infty_TAL^2_\tau}^2\|W_y\|_{L^2_TAH^s_\tau}\\
&\lesssim \tfrac M{K^2}\bigl(\|U\|_{L^2_TAH^{s+\frac12}_\tau} + \|W_y\|_{L^2_TAH^s_\tau}\bigr),
\end{align*}
uniformly in \(j\).
Applying Proposition~\ref{prop:ModelAP} we then obtain
\begin{align*}
\sup\limits_{t\in[0,T]}\|W_y\|_{AH^{s - \frac12}_\tau} + \sqrt{M\tau_0}\|W_y\|_{L^2_TAH^s_\tau} &\lesssim \|W_{0y}\|_{AH^{s-\frac12}_{\tau_0}} + \bigl(\tfrac1{K\tau_0\sqrt{M}} + \tfrac1{K^2\tau_0}\bigr)\sqrt{M\tau_0}\|W_y\|_{L^2_TAH^s_\tau} \\
&\quad + \tfrac1{K^2}\|W_y\|_{L^\infty_TH^{s - \frac12}_\tau} + \tfrac1{K^2\tau_0}\sqrt{M\tau_0}\|U\|_{L^2_TH^s_\tau},
\end{align*}
uniformly in \(j\).
Provided \(K\gg1\) is sufficiently large (depending only on \(\delta,s,\tau_0\)), we may then apply the estimate \eqref{U-AP-Bound} to yield
\begin{equation}\label{Wy-AP-Bound}
\sup\limits_{t\in[0,T]}\|W_y\|_{AH^{s - \frac12}_\tau} + \sqrt{M\tau_0}\|W_y\|_{L^2_TAH^s_\tau} \lesssim_K \|U_0\|_{AH^s_{\tau_0}} + \|W_{0y}\|_{AH^{s-\frac12}_{\tau_0}},
\end{equation}
uniformly in \(M\geq 1\) and \(j\).  In particular, provided \(M\gg1\) is sufficiently large (depending on \(U_0,W_0,K,\delta,s,\tau_0\)) we may ensure that
\[
\sup\limits_{t\in[0,T]}\|W_y\|_{AH^{s - \frac12}_\tau} + \sqrt{M\tau_0}\|W_y\|_{L^2_TAL^2_\tau} \leq \tfrac{\sqrt M}{40K},
\]
which closes the next part of the bootstrap.

For the final part of the bootstrap, we first apply Bernstein's inequality and \eqref{Wy-AP-Bound} to bound
\begin{equation}\label{HF-W-Bound}
\sup\limits_{t\in[0,T]}\|P_{>0}W\|_{AL^\infty_\tau}\lesssim \sup\limits_{t\in[0,T]}\|W_y\|_{AH^{s-\frac12}_\tau}\lesssim_K \|U_0\|_{AH^s_{\tau_0}} + \|W_{0y}\|_{AH^{s-\frac12}_{\tau_0}},
\end{equation}
provided \(K\gg1\) is sufficiently large, uniformly in \(M\geq 1\) and \(j\). It remains to bound the low frequencies, for which we will use Proposition~\ref{prop:TransportModelAP}. We first apply Bernstein's inequality, the estimate \eqref{Wy-AP-Bound}, and the fact that \(T\lesssim \frac1 M\) to bound
\[
\|P_0P_{\leq j}W_{\leq j,yy}\|_{L^1_TAL^\infty_\tau} \lesssim T\|W_y\|_{L^\infty_TAH^{s-\frac12}_\tau} \lesssim_K \tfrac1M\bigl(\|U_0\|_{AL^s_{\tau_0}} + \|W_{0y}\|_{AH^{s-\frac12}_{\tau_0}}\bigr),
\]
provided \(K\gg1\) is sufficiently large, uniformly in \(M\geq 1\) and \(j\). Similarly, we may bound
\begin{align*}
&\|P_0(2W_{\leq j}^2-\tfrac12|W_{\leq j}|^2)_y\|_{L^1_TAL^\infty_\tau} \lesssim \tfrac1M\|W\|_{AL^\infty_\tau}^2\lesssim \tfrac1{K\sqrt M}\|W\|_{L^\infty_TAL^\infty_\tau},\\
&\|P_0(\alpha_{\leq j}\beta_{\leq j} W_{\leq j})\|_{L^1_TAL^\infty_\tau} \lesssim T\|W\|_{L^\infty_TAL^\infty_\tau}^3\lesssim \tfrac1{K^2}\|W\|_{L^\infty_TAL^\infty_\tau},\\
&\|P_0(|U_{\leq j}|^2\alpha_{\leq j})\|_{L^1_TAL^\infty_\tau} \lesssim\||U|^2\alpha\|_{L^1_TAL^1_\tau}\lesssim T\|U\|_{L^\infty_TAL^2_\tau}^2\|W\|_{L^\infty_TAL^\infty_\tau}\lesssim \tfrac1{K^2}\|W\|_{L^\infty_TAL^\infty_\tau},
\end{align*}
where all the estimates are uniform in \(j\).
Applying these bounds, Proposition~\ref{prop:TransportModelAP}, and the estimates \eqref{U-AP-Bound} and \eqref{HF-W-Bound} we obtain
\begin{align*}
\sup\limits_{t\in[0,T]}\|W\|_{AL^\infty_\tau}&\lesssim \sup\limits_{t\in[0,T]}\|P_0W\|_{AL^\infty_\tau} + \sup\limits_{t\in[0,T]}\|P_{>0}W\|_{AL^\infty_\tau}\\
&\lesssim \|W_0\|_{AL^\infty_\tau} + C(K)\Bigl(\|U_0\|_{AH^s_{\tau_0}} + \|W_{0y}\|_{AH^{s-\frac12}_{\tau_0}}\Bigr)\\
&\quad + \Bigl(\tfrac1{K\sqrt M} + \tfrac1{K^2}\Bigr)\|W\|_{L^\infty_TAL^\infty_\tau},
\end{align*}
uniformly in \(j\), where \(C(K)\) is a constant depending only on \(K,\delta,s,\tau_0\). Once again, provided \(K\gg1\) is sufficiently large (depending only on \(\delta,s,\tau_0\)) we obtain the estimate
\begin{equation}\label{W-AP-Bound}
\sup\limits_{t\in[0,T]}\|W\|_{AL^\infty_\tau}\lesssim_K \|U_0\|_{AH^s_{\tau_0}} + \|W_0\|_{AZ^s_{\tau_0}},
\end{equation}
uniformly in \(M\geq 1\) and \(j\). Consequently, provided \(M\gg1\) is sufficiently large (depending only on \(U_0,W_0,K,s,\delta,\tau_0\)) we obtain the bound
\[
\sup\limits_{t\in[0,T]}\|W\|_{AL^\infty_\tau}\leq \tfrac{\sqrt M}{40K},
\]
which suffices to close the bootstrap.

With the bootstrap closed, the estimate \eqref{U-AP} for solutions of \eqref{RegularizedU} follows from \eqref{U-AP-Bound} and the estimate \eqref{W-AP} for solutions of \eqref{RegularizedW}, from \eqref{Wy-AP-Bound} and \eqref{W-AP-Bound}, where both estimates are uniform in \(j\). Passing to a subsequence as \(j\to \infty\), we may extract a weak* limit \(U\in \Cont_w([0,T];AH^s_\tau)\) and \(W\in \Cont_w([0,T];AZ^s_\tau)\) satisfying the equations \eqref{U-eqn}, \eqref{W-eqn} and the estimates \eqref{U-AP}, \eqref{W-AP}. Finally, the estimate \eqref{T-bound} follows from observing that for a suitable (very large) implicit constant, independent of \(U_0,W_0\), we may take
\[
M\lesssim \|U_0\|_{AH^s_{\tau_0}}^2 + \|W_0\|_{AZ^s_{\tau_0}}^2.
\]
\end{proof}

\section{Proof of Theorem~\ref{thrm:Main}}\label{sect:Invert}
We now turn to the proof of Theorem~\ref{thrm:Main}. Due to the existence of a solution to the system \eqref{U-eqn}, \eqref{W-eqn} proved in Theorem~\ref{thrm:UW-existence}, our main tasks in this section will first be to invert the change of variables \eqref{Coords} and then be to understand the regularity of our solution at the endpoints \(\pm x_0\) of the interval \(I\). 

\subsection{Existence} Given initial data \(u_0\in S\) satisfying our hypotheses, we denote the initial change of variables by
\[
y_0(x) = \int_0^x \frac1{|u_0(\zeta)|}\,d\zeta,
\]
which we recall is a well-defined diffeomorphism from \(I\) onto \(\R\). We then define the functions \(U_0,W_0\colon \R\to\C\) by
\[
U_0(y_0(x)) = u_0(x)\quad\text{and}\quad W_0(y_0(x)) = \frac{\bar u_0(x)u_{0x}(x)}{|u_0(x)|}.
\]
As \(u_0\in S\), there exists some \(0<s\leq \frac12\) and \(0<\tau_0\leq 1\) so that \(U_0\in AH^s_{\tau_0}\) and \(W_0\in AZ^s_{\tau_0}\). We may then apply Theorem~\ref{thrm:UW-existence} to obtain a solution of the system \eqref{U-eqn}, \eqref{W-eqn} on some time interval \([0,T]\).

A priori, the solution \((U,W)\) constructed in Theorem~\ref{thrm:UW-existence} is a distributional solution of the system \eqref{U-eqn}, \eqref{W-eqn}. However, by Sobolev embedding, for any \(n\geq 0\) and \(t\in[0,T]\) the spaces \(AH^s_\tau\) and \(AZ^s_\tau\) are both embedded in the space of bounded \(C^n\) functions. In particular, for any multi-index \(\kappa\in\N^2\) the functions \(\nabla_{t,y}^\kappa U,\nabla_{t,y}^\kappa W\) are continuous and bounded on \([0,T]\times \R\); for any \(t\in[0,T]\) the functions \(\nabla_{t,y}^\kappa U, \nabla_{t,y}^\kappa \partial_yW\) vanish as \(|y|\to\infty\); and \(U,W\) are classical solutions of \eqref{U-eqn}, \eqref{W-eqn} on \([0,T]\times \R\).

Applying Lemma~\ref{lem:Picard}, for any \(y\in \R\) we may find a solution of the ODE
\[
\begin{cases}
Y_t(t,y) = b(t,Y(t,y)) - \beta(t,Y(t,y)),\smallskip\\
Y(0,y) = y,
\end{cases}
\]
and by differentiating, it is clear that \(Y\colon[0,T]\times \R\to \R\) is smooth. In particular, \(c(t) = Y(t,0)\) is a solution of \eqref{GaugeFixing}.

As \(U\) is smooth, \(|U|\) is Lipschitz. Further, as \(U\) solves \eqref{U-eqn}, we may compute that
\begin{equation}\label{L1-Conserved}
(|U|)_t + \bigl((b - \beta)|U|\bigr)_y = 0.
\end{equation}
In particular,
\[
\partial_t\Bigl( Y_y(t,y) |U(t,Y(t,y))|\Bigr) = 0.
\]
Using the estimate \eqref{ODE-growth} and that \(|U_0(y)|>0\) for every \(y\in \R\), we obtain
\begin{equation}\label{y-diffeo}
|U(t,y)|>0\qtq{for every}(t,y)\in[0,T]\times \R,
\end{equation}
and consequently \(|U|\) is smooth. Further, as \(|U(t,y)|\to0\) as \(|y|\to\infty\), we may use the equations \eqref{L1-Conserved}, \eqref{GaugeFixing} to obtain
\[
\frac d{dt}\int_{c(t)}^\infty |U(t,\zeta)|\,d\zeta = 0 = \frac d{dt}\int_{-\infty}^{c(t)}|U(t,\zeta)|\,d\zeta.
\]
From the definition of \(U(0,y)\), we then have
\begin{equation}\label{y-fixed}
\int_{c(t)}^\infty |U(t,\zeta)|\,d\zeta = x_0 = \int_{-\infty}^{c(t)}|U(t,\zeta)|\,d\zeta.
\end{equation}

Next we define
\[
x(t,y) := \int_{c(t)}^y |U(t,\zeta)|\,d\zeta,
\]
and from \eqref{y-diffeo}, \eqref{y-fixed} we see that for any \(t\in[0,T]\) the map \(y\mapsto x(t,y)\) is a smooth diffeomorphsim from \(\R\) onto \(I\). Further, by construction, the map \(y\mapsto x(0,y)\) is the inverse of \(y_0\). We then take
\[
u(t,x) := \begin{cases}U(t,y)&\quad\text{if}\quad x = x(t,y)\in I,\\0&\quad\text{if}\quad x\not\in I.\end{cases}
\]
As the map \((t,y)\mapsto x(t,y)\) is smooth as a map from \([0,T]\times \R\to I\) and \(U\) is a smooth solution of \eqref{U-eqn}, the function \(u\) is a smooth solution of \eqref{QLS} on \([0,T]\times I\) and satisfies \(u(0,x) = u_0(x)\).

It remains to verify that \(u\) is sufficiently well-behaved at the endpoints \(\pm x_0\) to solve \eqref{QLS} on \([0,T]\times \R\). We first note that as \(|U|\to 0\) as \(|y|\to\infty\) we have that \(u\in \Cont([0,T]\times \R)\) is smooth away from \(x = \pm x_0\). Further, for \(x = x(t,y)\) we have
\[
u_x = \frac{U}{|U|}W\qtq{and}\bigl(\tfrac12 u^2\bigr)_{xx} = \frac{U^2}{|U|^2}\Bigl(2W^2 - \alpha W + W_y\Bigr).
\]
In particular, \(u_x,(\frac12 u^2)_{xx}\in \Cont([0,T];L^\infty(\R))\) are smooth away from \(x = \pm x_0\). This suffices to show that \(u\) solves \eqref{QLS}, where both sides of the equation make sense as continuous functions on \([0,T]\times \R\).

\subsection{Conservation laws} By construction, the solution \(u\) conserves its support. To prove that it also conserves its mass, momentum, and energy it suffices to show that our solution has sufficient regularity to justify the integrations by parts.

For the conservation of mass we only require that that \(u\in \Cont([0,T]\times \R)\) is supported on \(\bar I\), and that \(u_x,(\tfrac12 u^2)_{xx}\in \Cont([0,T];L^\infty(\R))\).

For the conservation of momentum and energy, we also require that \(\bigl[\bar u(\frac12 u^2)_{xx}\bigr]_x\in \Cont([0,T];L^\infty(\R))\). However, this follows from the observation that for \(x = x(t,y)\) we have
\[
\Bigl[\bar u\bigl(\tfrac12 u^2\bigr)_{xx}\Bigr]_x = \frac U{|U|}\Bigl(4WW_y - \alpha W_y - \alpha_yW + W_{yy} + 2W^3\Bigr).
\]

\subsection{Uniqueness and continuity of the solution map} Finally, we consider the problem of uniqueness and continuity of the solution map in \(L^2\). This follows from a straightforward energy estimate:
\begin{lem}
Suppose that for some \(T>0\) and \(j = 1,2\) the functions \(u\sbrack{j}\in \Cont([0,T]\times \R)\) are solutions of \eqref{QLS} with initial data \(u\sbrack{j}(0) = u\sbrack{j}_0\) that are non-zero and smooth on \(I\), supported on \(\bar I\), and such that \(
u\sbrack j_x,\Bigl[\tfrac12\bigl(u\sbrack j\bigr)^2\Bigr]_{xx}\in \Cont([0,T]; L^\infty(\R))\). Then we have the estimate
\begin{equation}\label{LipschitzBound}
\sup\limits_{t\in[0,T]}\|u\sbrack{1} - u\sbrack{2}\|_{L^2}\lesssim \|u\sbrack{1}_0 - u\sbrack{2}_0\|_{L^2}.
\end{equation}
\end{lem}
\begin{proof}
We first note that under our hypotheses on \(u\sbrack j\) we may justify the integration by parts
\[
\<iu\sbrack j_t,u\sbrack k\> = - \tfrac12\<[( u\sbrack j)^2]_x, (u\sbrack ju\sbrack k)_x\> + \mu\< ( u\sbrack j)^2,u\sbrack ju\sbrack k\>.
\]
Consequently, taking \(w = u\sbrack{1} - u\sbrack{2}\) and using the conservation of mass, we obtain the identity
\[
\frac d{dt} \|w\|_{L^2}^2 = - \tfrac12\Im\<a_x,(w^2)_x\> + \mu \Im\<a,w^2\>,
\]
where
\[
a = (u\sbrack1)^2 + (u\sbrack2)^2.
\]
Using that \(a_x\in \Cont([0,T]\times \R)\) and \(a_{xx}\in \Cont([0,T];L^\infty(\R))\), we may integrate by parts to obtain the estimate
\[
\frac d{dt}\|w\|_{L^2}^2 \lesssim \|a\|_{W^{2,\infty}}\|w\|_{L^2}^2,
\]
and the estimate \eqref{LipschitzBound} then follows from Gronwall's Inequality.
\end{proof}

This completes the proof of Theorem~\ref{thrm:Main}.\qed

\section{Stability of the compact breather}\label{sect:STABILITY}
In this section we prove Theorem~\ref{thrm:OrbitalStability}.  We explore the concentration compactness approach from Cazenave-Lions \cite{cazenave1982orbital} in the context of compacton stability. Rather than working with the variable \(u\) it will be useful to instead work with \(q := u^2\). By a slight abuse of notation, in this section we will denote the Hamiltonian and mass in terms of \(q\) instead of \(u\), i.e.
\[
H[q] = \tfrac 14 \int |q_x|^2\,dx - \tfrac12\int |q|^2\,dx,\qquad M[q] = \int |q|\,dx.
\]
For simplicity let us take $\omega =1$ and $\phi = \phi_1$ from \eqref{compacton}. We denote the orbit of the square of the compacton by
\[
E:= \left\{e^{2i\theta}\phi(\cdot - h)^2:\theta,h\in \R\right\}.
\]
The following result follows from the analysis of~\cite{MR4042218}:

\begin{prop}\label{prop:OldPaperConsequence}~
\begin{enumerate}
\item If \(q\in L^1\cap \dot H^1\) is a solution of the minimization problem
\begin{equation}\label{MinimizationProblem}
\min H[q]\quad\text{subject to}\quad M[q] = \sqrt{2\pi}
\end{equation}
then \(q\in E\).
\smallskip
\item Given any sequence \(\{q\sbrack{n}\}\subset L^1\cap \dot H^1\) such that \(M[q\sbrack{n}]\rightarrow\sqrt{2\pi}\) and \(H[q\sbrack{n}]\rightarrow H[\phi^2]\), there exists a sequence \(\{h\sbrack n\}\subseteq \R\) so that the sequence \(\{q\sbrack n(\cdot + h\sbrack n)\}\) is relatively compact in \(L^1\cap \dot H^1\).
\end{enumerate}
\end{prop}
\begin{proof}
The proof of (1) follows from the remark at the beginning of Section 3.2 in \cite{MR4042218}.  The proof of (2) follows by a slight adaptation of the proof of Theorem 2.7 in \cite{MR4042218}.
\end{proof}

\begin{proof}[Proof Theorem~\ref{thrm:OrbitalStability}] We proceed by contradiction.  Namely, assume that our orbital stability result does not hold. Then there exists $\epsilon > 0$ and a sequence of initial data $\{u_0\sbrack{n}\}\subseteq S$ so that after applying Theorem~\ref{thrm:Main} we obtain corresponding solutions \(u\sbrack n\) defined on the time interval \([0,T_n]\), such that, with \(q_0\sbrack n = (u_0\sbrack n)^2\) and \(q\sbrack n = (u\sbrack n)^2\), we have $M[q_0\sbrack{n}] \to \sqrt{2 \pi}$ and $H [q_0\sbrack{n}] \to H[\phi^2]$, and times \(0\leq t_n\leq T_n\) so that
\[
\inf_{\psi \in E}\|q\sbrack{n}(t_n,\cdot) - \psi\|_{L^1\cap \dot H^1}\geq \epsilon.
\]
As $q\sbrack{n} (t_n,\cdot)$ is a minimizing sequence for the constrained minimization problem \eqref{MinimizationProblem}, and using the fact that the solution \(q\sbrack n\) conserves the mass and energy, we may apply Proposition~\ref{prop:OldPaperConsequence} to obtain a contradiction.
\end{proof}

\begin{appendix}

\section{Changes of variable}\label{app:Computations}
In this section we outline the computations leading to the equations \eqref{U-eqn} for \(U\) and \eqref{W-eqn} for \(W\).

We first observe that
\[
\left(|u|^2\right)_t = 2\Im\left(|u|^2\bar uu_x\right)_x.
\]
Differentiating the expression \eqref{Coords} and using the equation \eqref{GaugeFixing} to replace \(c_t\) we obtain
\begin{align*}
y_t(t,x) &= - 3\int_0^x \frac1{|u|} \Re\left(\frac{\bar uu_\zeta}{|u|}\right) \Im\left(\frac{\bar uu_\zeta}{|u|}\right)\,d\zeta  - \int_0^x \Im\left(\frac{\bar uu_\zeta}{|u|}\right)_\zeta\,d\zeta + c_t(t)\\
&=  - 3\int_{c(t)}^{y(t,x)} \alpha(t,\zeta)\beta(t,\zeta)\,d\zeta + \beta(t,c(t)) - \beta(t,y(t,x)) + c_t(t)\\
&= b(t,y(t,x)) - \beta(t,y(t,x)).
\end{align*}
Consequently,
\[
\partial_t\bigl[U(t,y(t,x))\bigr] = \bigl[U_t + (b - \beta)U_y\bigr](t,y(t,x)).
\]

The equation \eqref{U-eqn} then follows from the observation that
\begin{align*}
u_x(t,x) &= \left[\frac1{|U|}U_y\right](t,y(t,x)),\\
u_{xx}(t,x) &= \left[\frac1{|U|^2}\left(U_{yy} - \alpha U_y\right)\right](t,y(t,x)).
\end{align*}

To derive the equation \eqref{W-eqn} we first compute that
\[
i\left(\frac{\bar uu_x}{|u|}\right)_t = \frac{\bar u}{|u|}\left(\bar u(uu_x)_x + \mu|u|^2u\right)_x - \frac{u_x}{|u|}\left(u(\bar u\bar u_x)_x + \mu|u|^2\bar u\right) - i\frac{\bar uu_x}{|u|^3}\Im\left(|u|^3\frac{\bar uu_x}{|u|}\right)_x.
\]
Changing variables and using that \(U_y = WU\) we obtain
\begin{align*}
i(W_t + (b - \beta)W_y) &= \frac{\bar U}{|U|^2}\left(\frac{\bar U}{|U|}\left(\frac{U^2}{|U|}W\right)_y + \mu|U|^2U\right)_y - \frac{U}{|U|^2}W\left(\frac{U}{|U|}\left(\frac{\bar U^2}{|U|}\bar W\right)_y + \mu |U|^2\bar U\right)\\
&\quad - i\frac W{|U|^3}\left(|U|^3\beta\right)_y.
\end{align*}
The equation \eqref{W-eqn} then follows from the fact that \(U_y = UW\) and \(\left(|U|\right)_y = \alpha |U| \).

\section{Multilinear estimates}\label{app:Multilinear}
In this section we prove Proposition~\ref{prop:NonlinearEstimates}. We will rely on the following lemma, the proof of which may be found in e.g.~\cite[Chapter 3 and Appendix D]{MR1121019}:
\begin{lem}
For \(s\in \R\), \(\sigma\geq0\) and any integer \(k\geq 0\) we have the estimates
\begin{align}
\|\bT_{\partial_y^kf} g\|_{H^s} &\lesssim \|f\|_{L^\infty}\|\partial_y^kg\|_{H^s},\label{LH-1}\\
\|\bT_{\<D_y\>^\sigma f}g\|_{H^s} &\lesssim \|f\|_{L^2}\|\<D_y\>^{s+\sigma}g\|_{\BMO},\label{LH-2}\\
\|\bPi[f,g]\|_{H^\sigma} &\lesssim \|\<D_y\>^{-s}f\|_{\BMO}\|g\|_{H^{\sigma+s}}.\label{HH-1}
\end{align}
\end{lem}

\begin{proof}[Proof of Proposition~\ref{prop:NonlinearEstimates}]~

\noindent\underline{Proof of \eqref{Symm-1}.} We decompose the product
\[
fg_y = \bT_fg_y + \bT_{g_y}f + \bPi[f,g_y],
\]
and then apply the estimates \eqref{LH-1}, \eqref{HH-1}.

\noindent\underline{Proof of \eqref{Symm-2}.} We first bound
\[
\|fg\|_{L^\infty}\leq \|f\|_{L^\infty}\|g\|_{L^\infty}.
\]
Next we decompose the product
\[
fg = \bT_fg + \bT_gf + \bPi[f,P_0g] + \bPi[f,P_{>0}g].
\]
For the low-high interactions we apply the estimate \eqref{LH-1} to bound
\begin{align*}
\|\partial_y\bT_fg\|_{H^{s-\frac12}}\lesssim \|f\|_{L^\infty}\|g_y\|_{H^{s-\frac12}},\\
\|\partial_y\bT_gf\|_{H^{s-\frac12}}\lesssim \|f_y\|_{H^{s-\frac12}}\|g\|_{L^\infty}.
\end{align*}
For the first high-high interaction, we observe that \(\bPi[f,P_0g] = P_{\leq 16}\bPi[P_{\leq 8}f,P_0g]\) so we may apply \eqref{HH-1} to bound
\begin{align*}
\|\partial_y\bPi[f,P_0g]\|_{H^{s-\frac12}} &\lesssim \|\bPi[P_{\leq 8}f_y,P_0g]\|_{L^2} + \|\bPi[P_{\leq 8}f,P_0g_y]\|_{L^2}\\
&\lesssim \|P_{\leq 8}f_y\|_{L^2}\|g\|_{L^\infty} + \|f\|_{L^\infty}\|P_0g_y\|_{L^2}\\
&\lesssim \|f_y\|_{H^{s-\frac12}}\|g\|_{L^\infty} + \|f\|_{L^\infty}\|g_y\|_{H^{s-\frac12}}.
\end{align*}
For the second high-high interaction, we again apply \eqref{HH-1} to bound
\[
\|\partial_y\bPi[f,P_{>0}g]\|_{H^{s-\frac12}}\lesssim \|f\|_{L^\infty}\|P_{>0}g\|_{H^{s+\frac12}}\lesssim \|f\|_{L^\infty}\|g_y\|_{H^{s-\frac12}}.
\]

\noindent\underline{Proof of \eqref{Asym-1}.} We simply apply the estimate \eqref{LH-1}.

\noindent\underline{Proof of \eqref{Asym-2}.} We decompose
\[
fg_y - \bT_fg_y = \bT_{g_y}f + \bPi[f,g_y].
\]
For the low-high interactions we apply the estimate \eqref{LH-2}, using that \(s\leq 1\), to obtain 
\[
\|\bT_{g_y}f\|_{H^s}\lesssim \|\<D_y\>f\|_{\BMO}\|g_y\|_{H^{s-1}}\lesssim \|f\|_{W^{1,\infty}}\|g\|_{H^s}.
\]
For the high-high interactions we apply the estimate \eqref{HH-1}, using that \(s\geq 0\), to obtain
\[
\|\bPi[f,g_y]\|_{H^s}\lesssim \|\<D_y\>f\|_{\BMO}\|g_y\|_{H^{s-1}}\lesssim \|f\|_{W^{1,\infty}}\|g\|_{H^s}.
\]

\noindent\underline{Proof of \eqref{Asym-3}.} Again we decompose
\[
fg_y = \bT_fg_y + \bT_{g_y}f + \bPi[P_0f,g_y] + \bPi[P_{>0}f,g_y].
\]
Applying \eqref{LH-1} we may bound
\[
\|\bT_fg_y\|_{H^{s-\frac12}}\lesssim \|f\|_{L^\infty}\|g_y\|_{H^{s-\frac12}},
\]
and applying \eqref{LH-2}, using that \(s\leq \frac12\), we may bound
\[
\|\bT_{g_y}f\|_{H^{s-\frac12}}\lesssim \|f\|_{\BMO}\|g_y\|_{H^{s-\frac12}} \lesssim \|f\|_{L^\infty}\|g_y\|_{H^{s-\frac12}}.
\]
For the first high-high interaction we use that \(\bPi[P_0f,g_y] = P_{\leq 16}\bPi[P_0f,P_{\leq 8}g_y]\) and apply \eqref{HH-1} to bound
\[
\|\bPi[P_0f,g_y]\|_{H^{s-\frac12}}\lesssim \|\bPi[P_0f,g_y]\|_{L^2}\lesssim \|f\|_{L^\infty}\|P_{\leq 8}g_y\|_{L^2}\lesssim \|f\|_{L^\infty}\|g_y\|_{H^{s-\frac12}}.
\]
For the second high-high interaction we apply Bernstein's inequality at the output frequency followed by the Cauchy-Schwarz inequality to obtain
\begin{align*}
\|\bPi[P_{>0}f,g_y]\|_{H^{s-\frac12}} &\lesssim \sum\limits_{j\geq 0}2^{sj}\|P_j\bPi[P_{>0}f,g_y]\|_{L^1}\\
&\lesssim \sum\limits_{|k-\ell|<4}\left(2^{\frac12k}\|P_kP_{>0}f\|_{L^2}\right)\left(2^{(s-\frac12)\ell}\|P_\ell g_y\|_{L^2}\right)\\
&\lesssim \|P_{>0}f\|_{H^{\frac12}}\|g_y\|_{H^{s-\frac12}}\lesssim \|f_y\|_{H^{-\frac12}}\|g_y\|_{H^{s-\frac12}},
\end{align*}
where we have used the fact that \(s>0\).

\noindent\underline{Proof of \eqref{Trilinear}.} Here it will be convenient to argue by duality. Consequently, we take a test function \(\phi\in H^{-s}\) and decompose by frequency to obtain
\[
\<fgh,\phi\> = \sum\limits_{j_1,j_2,j_3,j_4 \geq 0}\<P_{j_1}f\cdot P_{j_2}g\cdot P_{j_3}h,P_{j_4}\phi\>.
\]
By symmetry, we may assume that \(j_1\leq j_2\leq j_3\) and then observe that the sum vanishes unless \(|\max\{j_2,j_4\} - j_3|\leq 8\).

If \(j_4\geq j_2\) then \(|j_3 - j_4|\leq 8\). We estimate the two lowest frequency terms in \(L^\infty\) and the two highest frequency terms in \(L^2\) and then apply Bernstein's inequality to bound
\begin{align*}
|\<P_{j_1}f\cdot P_{j_2}g\cdot P_{j_3}h,P_{j_4}\phi\>| &\lesssim \|P_{j_1}f\|_{L^\infty}\|P_{j_2}g\|_{L^\infty}\|P_{j_3}h\|_{L^2}\|P_{j_4}\phi\|_{L^2}\\
&\lesssim 2^{\frac12(j_1 - j_3)}\|P_{j_1}f\|_{L^2}\|P_{j_2}g\|_{H^{\frac12}}\|P_{j_3}h\|_{H^{s+\frac12}}\|P_{j_4}\phi\|_{H^{-s}}.
\end{align*}
If \(j_4<j_2\) then \(|j_2 - j_3|\leq 8\) we proceed similarly, applying Bernstein's inequality to estimate
\begin{align*}
|\<P_{j_1}f\cdot P_{j_2}g\cdot P_{j_3}h,P_{j_4}\phi\>| &\lesssim \|P_{j_1}f\|_{L^\infty}\|P_{j_2}g\|_{L^2}\|P_{j_3}h\|_{L^2}\|P_{j_4}\phi\|_{L^\infty}\\
&\lesssim 2^{(s+\frac12)(j_4-j_3) + \frac12 (j_1 - j_2)}\|P_{j_1}f\|_{L^2}\|P_{j_2}g\|_{H^{\frac12}}\|P_{j_3}h\|_{H^{s+\frac12}}\|P_{j_4}\phi\|_{H^{-s}}.
\end{align*}

The estimate \eqref{Trilinear} then follows from several applications of the Cauchy-Schwarz inequality, first summing over the lowest frequency, then the second lowest frequency, and finally the highest two (comparable) frequencies.

\noindent\underline{Proof of \eqref{Trilinear-2}.} Again taking \(\phi\in H^{-s}\) and decomposing by frequency, it suffices to bound
\[
|\<P_{j_1}f\cdot P_{j_2}g\cdot h,P_{j_4}\phi\>|,
\]
where, by symmetry, we may assume that \(j_1\leq j_2\). We then divide into 3 cases:

\emph{Case 1: \(j_2\leq j_4\).} For the low frequency part of \(h\) we apply Bernstein's inequality to bound
\begin{align*}
|\<P_{j_1}f\cdot P_{j_2}g\cdot P_{\leq j_2}h,P_{j_4}\phi\>| &\lesssim \|P_{j_1}f\|_{L^\infty}\|P_{j_2}g\|_{L^2}\|P_{\leq j_2}h\|_{L^\infty}\|P_{j_3}\phi\|_{L^2}\\
&\lesssim 2^{\frac12(j_1 - j_2)}\|P_{j_1}f\|_{L^2}\|P_{j_2}g\|_{H^{s+\frac12}}\|P_{\leq j_2}h\|_{L^\infty}\|P_{j_4}\phi\|_{H^{-s}},
\end{align*}
where we have used that \(|j_2 - j_4|\leq 8\). For the high frequency part we decompose by frequency \(j_3>j_2\) and bound
\begin{align*}
|\<P_{j_1}f\cdot P_{j_2}g\cdot P_{j_3}h,P_{j_4}\phi\>|&\lesssim \|P_{j_1}f\|_{L^\infty}\|P_{j_2}g\|_{L^\infty}\|P_{j_3}h\|_{L^2}\|P_{j_4}\phi\|_{L^2}\\
&\lesssim 2^{\frac12(j_1 - j_3) + \frac12(j_2 - j_3)}\|P_{j_1}f\|_{L^2}\|P_{j_2}g\|_{L^2}\|P_{j_3}h_y\|_{H^s}\|P_{j_4}\phi\|_{H^{-s}},
\end{align*}
where we have used that \(j_3>0\) and \(|j_3 - j_4|\leq 8\).

\emph{Case 2: \(j_1\leq j_4<j_2\).} Here we proceed similarly, bounding the low frequency part of \(h\) by
\begin{align*}
|\<P_{j_1}f\cdot P_{j_2}g\cdot P_{\leq j_4}h,P_{j_4}\phi\>| &\lesssim \|P_{j_1}f\|_{L^\infty}\|P_{j_2}g\|_{L^2}\|P_{\leq j_4}h\|_{L^\infty}\|P_{j_4}\phi\|_{L^2}\\
&\lesssim 2^{\frac12(j_1 - j_2)}\|P_{j_1}f\|_{L^2}\|P_{j_2}g\|_{H^{s+\frac12}}\|P_{\leq j_4}h\|_{L^\infty}\|P_{j_4}\phi\|_{H^{-s}},
\end{align*}
and the high frequency part, where \(j_3>j_4\), by
\begin{align*}
|\<P_{j_1}f\cdot P_{j_2}g\cdot P_{j_3}h,P_{j_4}\phi\>|&\lesssim \|P_{j_1}f\|_{L^\infty}\|P_{j_2}g\|_{L^2}\|P_{j_3}h\|_{L^2}\|P_{j_4}\phi\|_{L^\infty}\\
&\lesssim 2^{\frac12(j_1-j_3) + (s+\frac12)(j_4 - j_3)}\|P_{j_1}f\|_{L^2}\|P_{j_2}g\|_{L^2}\|P_{j_3}h_y\|_{H^s}\|P_{j_4}\phi\|_{H^{-s}},
\end{align*}

\emph{Case 3: \(j_4<j_1\).} Proceeding as in the previous two cases we bound
\begin{align*}
|\<P_{j_1}f\cdot P_{j_2}g\cdot P_{\leq j_1}h,P_{j_4}\phi\>| &\lesssim \|P_{j_1}f\|_{L^2}\|P_{j_2}g\|_{L^2}\|P_{\leq j_1}h\|_{L^\infty}\|P_{j_4}\phi\|_{L^\infty}\\
&\lesssim 2^{(s+\frac12)(j_4 - j_2)}\|P_{j_1}f\|_{L^2}\|P_{j_2}g\|_{H^{s+\frac12}}\|P_{\leq j_4}h\|_{L^\infty}\|P_{j_4}\phi\|_{H^{-s}},
\end{align*}
and for \(j_3>j_1\)
\begin{align*}
|\<P_{j_1}f\cdot P_{j_2}g\cdot P_{j_3}h,P_{j_4}\phi\>|&\lesssim \|P_{j_1}f\|_{L^\infty}\|P_{j_2}g\|_{L^2}\|P_{j_3}h\|_{L^2}\|P_{j_4}\phi\|_{L^\infty}\\
&\lesssim 2^{(s + \frac12)(j_4 - j_3) + \frac12(j_1 - j_3)}\|P_{j_1}f\|_{L^2}\|P_{j_2}g\|_{L^2}\|P_{j_3}h_y\|_{H^s}\|P_{j_4}\phi\|_{H^{-s}}.
\end{align*}

\noindent\underline{Proof of \eqref{Comm}.} We decompose the commutator as
\begin{align*}
[\<D_y\>^s,f]g_y &= [\<D_y\>^s,\bT_f]g_y + \Bigl(\<D_y\>^s\bPi[P_0f,g_y] - \bPi[P_0f,\<D_y\>^sg_y]\Bigr)\\
&\quad + \<D_y\>^s\bT_{g_y}f - \bT_{\<D_y\>^sg_y}f + \<D_y\>^s\bPi[P_{>0}f,g_y] - \bPi[P_{>0}f,\<D_y\>^sg_y].
\end{align*}
For the first term we write
\begin{align*}
[\<D_y\>^s,P_{\leq j-4}f]P_jg_y &= [\<D_y\>^sP_{\leq j+4},P_{\leq j-4}f]P_jg_y\\
&= \int K_j(y-z) \bigl(P_{\leq j-4}f(y) - P_{\leq j-4}f(z)\bigr)P_jg_z(z)\,dz,
\end{align*}
where \(K_j\) is the kernel of \(\<D_y\>^sP_{\leq j+4}\). We may then apply Young's inequality to bound
\[
\|[\<D_y\>^s,P_{\leq j-4}f]P_jg_y\|_{L^2} \lesssim \|K_j(y) y \|_{L^1}\|P_{\leq j-4}f_y\|_{L^\infty}\|P_jg_y\|_{L^2}\lesssim \|f_y\|_{L^\infty}\|P_jg\|_{H^s},
\]
where we have used that
\[
\|K_j(y)y\|_{L^1}\lesssim 2^{(s-1)j}.
\]
The second term is bounded similarly, using that
\[
\<D_y\>^s\bPi[P_0f,g_y] - \bPi[P_0f,\<D_y\>^sg_y] = \<D_y\>^sP_{\leq 16}\bPi[P_0f,P_{\leq 8}g_y] - \bPi[P_0f,\<D_y\>^sP_{\leq 16}P_{\leq 8}g_y],
\]
to obtain
\[
\|\<D_y\>^s\bPi[P_0f,g_y] - \bPi[P_0f,\<D_y\>^sg_y]\|_{L^2}\lesssim \|f_y\|_{L^\infty}\|P_{\leq 8}g\|_{L^2}\lesssim \|f_y\|_{L^\infty}\|g\|_{H^s}.
\]

For the remaining terms, we first apply the estimate \eqref{LH-2}, with the fact that \(s\leq 1\), to obtain
\[
\|\<D_y\>^s\bT_{g_y}f\|_{L^2} + 
\|\bT_{\<D_y\>^sg_y}f\|_{L^2}\lesssim \|f_y\|_{L^\infty}\|g\|_{H^s}.
\]
Next, we apply the estimate \eqref{HH-1} to bound
\[
\|\bPi[P_{>0}f,\<D_y\>^sg_y]\|_{L^2} \lesssim \|f_y\|_{L^\infty}\|g\|_{H^s}.
\]
For the remaining term, if \(s>0\), we again apply the estimate \eqref{HH-1} to bound
\[
\|\<D_y\>^s\bPi[P_{>0}f,g_y]\|_{L^2} \lesssim \|f_y\|_{L^\infty}\|g\|_{H^s},
\]
whereas, if \(-\frac12<s<0\), we argue as in the proof of \eqref{Asym-3} to bound
\[
\|\<D_y\>^s\bPi[P_{>0}f,g_y]\|_{L^2}\lesssim \|f_{yy}\|_{H^{-\frac12}}\|g\|_{H^s}.
\]

\noindent\underline{Proof of \eqref{Comm-2}.} We first observe that
\[
[P_{\leq j},f]P_{\leq j-4}g_y = [P_{\leq j},P_{>j-4}f]P_{\leq j-4}g_y,
\]
and hence we may bound
\[
\|[P_{\leq j},f]P_{\leq j-4}g_y\|_{L^\infty}\lesssim \|P_{> j-4}f\|_{L^\infty}\|P_{\leq j-4}g_y\|_{L^\infty}\lesssim \|f_y\|_{L^\infty}\|g\|_{L^\infty}.
\]
For the remaining term, we proceed as in the proof of \eqref{Comm} and write
\[
[P_{\leq j},f]P_{>j-4}g_y = \int K_j(y-z) \bigl(f(y) - f(z)\bigr)P_{>j-4}g_z(z)\,dz,
\]
where \(K_j\) is the kernel of \(P_{\leq j}\). The estimate then follows from Young's inequality.
\end{proof}

\end{appendix}

\bibliographystyle{abbrv}
\bibliography{NLS}
\bigskip
\end{document}